\documentclass[ 11pt]{amsart}
\usepackage{amsmath}%
\usepackage{amsthm}
\usepackage{amsfonts}%
\usepackage{amssymb}%
\usepackage{graphicx}
\usepackage{mathrsfs}
\usepackage{xcolor}
\usepackage{fullpage}
\usepackage[backref]{hyperref}
\usepackage{comment}
\usepackage{lineno}
\usepackage[english,algoruled,lined,noresetcount,norelsize]{algorithm2e}
\usepackage[abbrev,msc-links]{amsrefs} 
\newtheorem{theorem}{Theorem}
\theoremstyle{plain}

\newtheorem{claim}[theorem]{Claim}

\newtheorem{definition}[theorem]{Definition}

\newtheorem{fact}[theorem]{Fact}
\newtheorem{lemma}[theorem]{Lemma}

\newtheorem{remark}[theorem]{Remark}

\numberwithin{equation}{section}
\numberwithin{theorem}{section}
\def\cA{\mathcal{A}}

\def\cX{\mathcal{X}}
\def\cY{\mathcal{Y}}

\def\lF{\ell_F}
\def\lham{\ell_{\mathrm{ham}}}
\def\ham{\mathrm{ham}}

\def\eps{\varepsilon}

\def \degen{\mathrm{degen}}

\def \ext{\text{ext}}

\def \r{\gamma}

\def\Zf{\mathbb{Z}_f}
\def\NN{\mathbb{N}}
\def\RR{\mathbb{R}}
\def\vecone{\boldsymbol{1}}
\def\vecu{\boldsymbol{u}}

\def\vece{\boldsymbol{e}}

\let\phi=\varphi

\title{Factors and loose Hamilton cycles in sparse pseudo-random hypergraphs}
\date{\today}
\author{Hi\d{\^{e}}p H\`an}
 \author{Jie Han}  
 \author{Patrick Morris}
\thanks{HH: Departamento de  Matem\'atica y Ciencia de la Computaci\'on, Universidad de Santiago de Chile, Las Sophoras 173, Santiago, Chile, {\tt hiep.han@usach.cl}. Research supported by FONDECYT Regular grant 1191838. \\
JH: Department of Mathematics, University of Rhode Island, 5 Lippitt Road, Kingston, RI, USA, 02881, {\tt jie\_han@uri.edu.} Research partially supported by Simons Foundation \#630884.
\\
PM: Institut f\"ur Mathematik, Freie Universit\"at Berlin, Arnimallee 3, 14195 Berlin, Germany and Berlin Mathematical School, Germany, {\tt pm0041@math.fu-berlin.de}. Research supported {in part by a Leverhulme Trust Study Abroad Studentship (SAS-2017-052\textbackslash 9) and by the Deutsche Forschungsgemeinschaft (DFG, German Research
Foundation) under Germany's Excellence Strategy - The Berlin Mathematics
Research Center MATH+ (EXC-2046/1, project ID: 390685689).}}
\date{\today}
\begin{document}

\maketitle

\begin{abstract}
We investigate the emergence of subgraphs in sparse pseudo-random $k$-uniform hypergraphs, using 
the following comparatively weak notion of pseudo-randomness.
A $k$-uniform hypergraph~$H$ on $n$ vertices is called \emph{$(p,\alpha,\eps)$-pseudo-random} if for all (not necessarily disjoint) vertex  subsets $A_1,\dots, A_k{\subseteq} V(H)$
with $|A_1|\cdots |A_k|{\geq}\alpha n^{k}$ we have \[{e(A_1,\dots, A_k)=(1\pm\eps)p |A_1|\cdots |A_k|.}\]
For any linear $k$-uniform $F$ we provide a bound on $\alpha=\alpha(n)$ in terms of $p=p(n)$ and $F$, 
such that (under natural divisibility assumptions on $n$) 
any $k$-uniform $\big(p,\alpha, o(1)\big)$-pseudo-random $n$-vertex hypergraph  $H$  with a mild minimum vertex degree condition
contains an $F$-factor. The approach also enables us to establish the existence of loose Hamilton cycles in sufficiently pseudo-random hypergraphs and, along the way, we also derive conditions which guarantee the appearance of any fixed sized subgraph. All 
results imply corresponding bounds for stronger notions of hypergraph pseudo-randomness such as jumbledness or large spectral gap.    % small second eigenvalue. 

As a consequence,  $\big(p,\alpha, o(1)\big)$-pseudo-random $k$-graphs as above {contain}: $(i)$ a perfect matching if  $\alpha=o(p^{k})$ and  $(ii)$ a   loose Hamilton cycle  if $\alpha=o(p^{k-1})$. 
 This extends the works of Lenz--Mubayi, and Lenz--Mubayi--Mycroft who studied the analogous problems in the dense setting.

\end{abstract}

%\end{titlepage}

%\linenumbers

\section{Introduction}
%The appearance of certain graphs $H$ as a subgraph of a graph $G$ is one of the most popular topic in the study of random and pseudo-random graphs.

%various areas of combinatorics. For example
%In the random graph model G(n, p) this question turned out comparatively easy for graphs H of constant size, but much harder for graphs H on n vertices, so-called spanning subgraphs. 

Pseudo-random graphs, vaguely speaking, are deterministic graphs which resemble their random counterparts in many characteristic properties.
%Interesting in its own right  their investigation  has also led to fundamental insights in many  fields of mathematics and computer science 
%including combinatorics, number theory, geometry, ergodic theory and algorithms and complexity. %, often  revealing deep connections between these areas.
The systematic study of the topic 
was initiated by  Andrew Thomason~\cite{Thomason1, Thomason2} who
introduced  a variant of
 the following notion of uniform edge distribution.
 %\footnote{Strictly speaking, Thomason worked with a  variant of this definition, where the condition holds only for all $A=B\subseteq V(G)$. Graphs satisfying the slightly stronger notion we give here are sometimes referred to as $(p,\beta)$-\emph{bijumbled} graphs in the literature.}. %which  edge distribution. 
 %Given $p=p(n)$ and $\beta=\beta(n)$  functions of $n$, 
A graph  $G=(V,E)$  is called \emph{$(p,\beta)$-jumbled} if  for all (not necessarily disjoint) $A,B{\subseteq} V$ we have
\footnote{Throughout the paper we write $x=y\pm z$ to denote that $y-z\leq x\leq y+z$.} 
\begin{equation}\label{eq:defjumbled}
e(A,B):= |\{(a,b)\in A\times B\colon \{a,b\}\in E\}|= p|A||B|\pm \beta \sqrt{|A||B|}. %\quad\text{holds for all disjoint }A,B{\subseteq} V,
\end{equation}
%where $e(A,B)$ denotes the 
%number of edges of $G$ with one end in $A$ and the other in $B$. 
%Vaguely speaking,
{The definition of} jumbledness captures  how close a graph is to having uniform edge distribution, with the parameter~$\beta$ controlling the discrepancy from this paradigm.  
Further, $\beta$  also controls  the size of subsets  $A,B{\subseteq} V$ for which the lower bound in~\eqref{eq:defjumbled}    
becomes void, namely,  once $|A||B|=o(\beta^2/p^2)$ holds.   
%Note that the smaller~$\beta$ is in terms of $n$ and~$p$ the smaller the deviation from the ``expected'' number of edges is.  
%  Smaller values of $\beta$ give a stronger pseudo-random condition. 
The random graph~$G(n, p)$ is $\big(p, O(\sqrt{pn})\big)$-jumbled almost always, which %\footnote{That is, with probability tending to $1$ as $n$ tends to infinity.}
 is essentially optimal since it follows from the proof of Erd\H{os} and Spencer~\cite{erdos1972imbalances} that any graph with edge density, say, $p<0.99$ satisfies $\beta=\Omega(\sqrt{pn})$. % {(this follows from the proof of~\cite{erdos1972imbalances}, see also \cite{KrivelevichSudakov}).}
% and so the random graph can be thought of as being optimally jumbled.
 
%When using the definition to prove results for jumbled graphs, it is most often this that determines the condition on $\beta$ which is required. 

%Alternative descriptions of pseudo-randomness exist, one of which is given in terms of the graph spectrum.
%This leads to an important and  well studied  subclass of jumbled graphs called $(n,d,\lambda)$-graphs. 
%Let  $\lambda_1\geq \lambda_2\geq\dots\geq\lambda_n$ denote the eigenvalues of the adjacency matrix of the graph $G$ and let
 %$\lambda(G)=\max\{|\lambda_2|,|\lambda_n|\}$ be its \emph{second eigenvalue}. 
%Then, the graph $G$ is an $(n,d,\lambda)$-graph if it is  a $d$-regular graph on $n$ vertices
%with $\lambda(G)\leq \lambda$. 
%By the Expander Mixing Lemma~\cite{Alon_exp, Tanner}  $(n,d,\lambda)$-graphs are known to be $(\frac dn,\lambda)$-jumbled. 
%Much of the work in the area of pseudo-random graph theory has been done using this notion. 

%, for example, 
%for most graphs $F$ we d

% case

%With high probability the random graph~$G(n, p)$ is $\big(p, O(\sqrt{pn})\big)$-jumbled  which is in a sense ``as pseudo-random as possible'' since 
%any $(p,\beta)$-jumbled graph with, say, $p\leq 0.9$ must satisfy $\beta=\Omega(\sqrt{pn})$~\cite{}. and in most cases understanding sparse pseudo-random graphs, i.e., when $p(n)=o(1)$, is  far more challenging than the corresponding dense case $p(n)=\Omega(1)$. 

%, i.e. that
%\begin{align}\label{eq:defndlambda}
%e(A,B)=\frac dn|A||B|\pm\lambda \sqrt{|A||B|}\quad\text{holds for all }A,B{\subseteq} V.
%\end{align}

One topic  of great importance and popularity in the area concerns   the appearance of certain subgraphs $F$ in sufficiently pseudo-random {graphs} $G$. 
% of pseudo-random graphs 
%concerns conditions/degree of pseudo-randomness which enforce the appearance of a certain graph $F$ in $G$ as subgraph.
%/the emergence of certain subgraph  $F$ in sufficiently pseudo-random graph $G$. % is a topic of
 Here,~$F$ can be a small, fixed size graph such as a triangle, an odd cycle {or}  a fixed size clique, or it can be a large, indeed spanning subgraph of~$G$ such as a
 perfect matching, a Hamilton cycle, 
 or a $K_r$-factor\footnote{That is, vertex disjoint copies of the $r$-clique $K_r$ covering all the vertices of $G$.}. 
 The fundamental question then concerns the degree of pseudo-randomness %(conditions on $\beta=\beta(n,p)$ in terms of $n$ and $p$  or conditions on~$\lambda=\lambda(n,d)$ in terms of $n$ and~$d$) 
 which ensures  
 that $F$ is a subgraph of $G$ and we 
distinguish here the (dense) \emph{quasi-random} case, when
 $p=\Omega(1)$ and $\beta=o(n)$, and 
the (sparse) \emph{pseudo-random} case,  when
 $p=o(1)$ and $\beta=\beta(n,p)$ is a function of $n$ and $p$.
% We note that the latter is more general than the former since 
Most of the time proofs of results in the latter case can easily be modified to cover the former as well and in this
sense problems concerning pseudo-randomness are typically more difficult than the corresponding ones for quasi-randomness.
For the quasi-random case the subgraph containment problem is well understood~\cite{CGW,komlos1997blow}; for the pseudo-random case, however, it
turned out to be  notoriously difficult  already  for small graphs~$F$, and even more so for spanning subgraphs.  
Thus, while  bounds exist for general graphs $F$ (see e.g.~\cite{KRS, sparseblowup}), only  few are known to be (essentially) best possible: triangles, odd cycles, perfect matchings, Hamilton cycles and 
triangle-factors~\cite{A94,AK,KrivelevichSudakov,Nenadov}.  
For further information on pseudo-random graphs and the related subgraph containment problem we refer the reader to the survey~\cite{KrivelevichSudakov}.

\subsection*{(Linear) pseudo-random hypergraphs}
%In this paper we investigate the corresponding question for hypergraphs. 
A $k$-uniform hypergraph, $k$-graph for short,
 is  a pair $H=(V,E)$ with a vertex set $V=V(H)$ and an edge set $E=E(H){\subseteq}\binom{V}k$, where $\binom{V}k$ denotes the set  of all $k$-element 
 subsets of $V$. 
  Launched by Chung and Graham~\cite{chung1990quasi}, the investigation of pseudo-random $k$-graphs is widely popular, albeit mostly restricted to the 
  dense case due to the complexity of the matter   \cite{quasihyper,Chung10,Ch90,Ch91,wquasi,G06,G07,KNRS,KRS02,LenzMubayi_eig,lenz2015poset,RS07a,RS07b,RSk04,Towsner}.
There are several generalisations of~\eqref{eq:defjumbled} to $k$-graphs,
 the simplest and most natural of which   is perhaps the following. %, which yields~\eqref{eq:defjumbled} for $k=2$. 
 A $k$-graph $H=(V,E)$ is called  \emph{$(p,\beta)$-jumbled} if 
 for all (not necessarily  disjoint) $A_1,\dots, A_k{\subseteq} V$ we have
 \begin{equation}\label{eq:defkjumbled}
e(A_1,\dots,A_k)=  p\prod_{i\in[k]}|A_i|\pm \beta {\prod_{i\in[k]}|A_i|}^{1/2}. %\quad\text{holds for all pairwise disjoint }A_1,\dots, A_k{\subseteq} V,
\end{equation}
where $e(A_1,\dots,A_k)$ denotes the number of tuples $(a_1,\dots,a_k)$, $a_i\in A_i$, which form an edge in $H$.

   Analogously to the graph setting, we separate {the} (dense)~\emph{quasi-random} case where
$p=\Omega(1)$ and $\beta= o(n^{k/2})$, a range for which~\eqref{eq:defkjumbled} only provides control over the edge distribution between linear size sets $A_i$, $i\in[k]$,
and the (sparse)~\emph{pseudo-random} case where $p=o(1)$ and $\beta=\beta(n,p)$.
  
As discussed in detail below, questions concerning quasi-random hypergraphs, in particular the emergence of subgraphs therein, have received a lot of attention  in recent years, 
see e.g.~\cite{KNRS, wquasi,LenzMubayi1,LMM,Friedman, FriedmanWigderson,HavTho, LenzMubayi_eig}.
These results about quasirandom hypergraphs required considerable effort and so it is no surprise that much less is known in the 
sparse (pseudo-random) setting.
%Due to the complicatedness, much less is known in the case of (sparse) pseudo-random hypergraphs. 
We refer to~\cite{KMST1,KMST2, FKL, conlon2014green} for results concerning subgraph containment using stronger and more complicated notions of sparse pseudo-randomness, most prominently the work by Conlon, Fox, Zhao~\cite{conlon2014green} reproving the Green-Tao Theorem.
 
%Mainly due to the complicatedness much less is known for the pseudo-random case. 

%It turns out that even in this range, the picture for hypergraphs is far more complex and many well known results from quasi-random graph theory do not generalise easily, if at all, to the hypergraph setting. 

\subsubsection*{Quasi-randomness and linear hypergraphs}
The subgraph containment problem for (dense) quasi-random $k$-graphs, $k\geq3$, 
has been an interesting topic and is a good example of how quasi-random hypergraphs can behave in a much more subtle manner than  graphs.
Indeed, with respect to subgraph statistics it is well known~\cite{CGW} that 
the number of labelled copies of any fixed size graph~$F$ in a large
 quasi-random graph with edge density {$p=\Omega(1)$} is roughly as expected from the random graph~$G(n,p)$.
 %, e.g.,
%in a sequence of $(p,o(n))$-pseudo-random (aka quasi-random) graph sequence $(G_n)_{n\to\infty}$,  $|V(G_n)|=n$  
%is  $p^{e_F}n^{v_F}+o(n^{v_F})$. %, where $e_F=|E(F)|$ and $v_F=|V(F)|$,
%thus approximately what is expected from the random graph $G(n,p)$. %with the same edge density as $G$.
%
%a quasi-randomness, aka $(p,o(n))$-pseudo-randomness
%is equivalent to the so called \emph{counting property} 
%, i.e., the sequence of graphs $(G_n)_{n\to\infty}$,  $n=|V(G_n)|$, is
% quasi-random  if and only if~$G_n$
%contains $p^{e_F}n^{v_F}+o(n^{v_F})$  labelled copies of \emph{any}  graph $F$, where $e_F=|E(F)|$ and $v_F=|V(F)|$.
%Moreover, the converse also hold, i.e., if the counting property holds
%(and it is enough to require it for the single graph $F=C_4$) then the graphs must be $(p,o(n))$-jumbled. 
For $k\geq 3$, however, R\"odl noted that by a construction of~\cite{ErdosHajnal}, quasi-random $k$-graphs may not contain a single copy of, 
say, a $(k+1)$-clique, let alone the expected number of such copies. % (see also~\cite{RRS1, RRS2} for further information).
In contrast and somewhat surprisingly, the  works of Kohayakawa {et al.}~\cite{KNRS} and Conlon {et al.}~\cite{wquasi}  show that quasi-randomness is strongly related, indeed \emph{equivalent},
to the counting property of \emph{linear} $k$-graphs,
those in which any  two edges  intersect in at most one vertex. 
More precisely, a sequence of  $k$-graphs~$(H_n)_{n\to\infty}$, $n=|V(H_n)|$, is $(p,o(n^{k/2}))$-jumbled (aka quasi-random)
if and only if $H_n$ contains $p^{e_F}n^{v_F}+o(n^{v_F})$  labelled copies of {any} linear $k$-graph $F$.
% satisfies the counting property
%for all  linear $k$-graphs $F$ of constant size and conversely, if the counting property  holds for all 
%linear $k$-graphs $F$ then $H$ must be quasi-random. 
Due to this reason,  quasi-randomness for $k$-graphs is often referred to as \emph{linear quasi-randomness}.
%This thus makes the class of linear $k$-graph the one naturally associated to~\eqref{eq:defkjumbled} in the dense qualitative case.
%We note that due to this discrepancy the counting property for general $k$-graph does not hold with this notion of quasi-randomness, 
%much research has looked into stronger notions of quasi-randomness  and counting properties for these hypergraphs. 
Much research has looked into stronger notions of uniform edge distribution and stronger counting properties. Their relations
have only been clarified  recently and we refer to~\cite{Towsner} (see also~\cite{quasihyper}) for further information. 
%Here, we concentrate on the `weak' notion of pseudo-randomness given by~\eqref{eq:defkjumbled}.

%As soon as $|A_i|=o(n)$ for some $i\in[k]$ the property~\eqref{eq:defkjumbled} becomes void.
%It should be noted that there is a hypergraph spectral property related to~\eqref{eq:defkjumbled}...
 % $\frac{|E(H)|}{\binom {|V|}k}=\Omega(1)$. 
  %and commonly referred to as quasi-random $k$-graphs,
%  Its investigation was launched by Chung and Graham [11] and driven by the works on Szemer\'edi's theorem then resolved Gowers [23] and R\"odl {et al.} [41, 39] .

%For this range of $p$ and $\beta$ it is known for $k=2$ that  $(p,\beta)$-jumbled $n$-vertex graphs satisfy the counting property, i.e., 
%there are $p^{e_F}n^{v_F}+o(n^{v_F})$  labelled copies of any constant size $F$ in these graphs, where $e_F=|E(F)|$ and $v_F=|V(F)|$.
%In this sense these graphs behave like expected from the random graph with the same edge density. Moreover, the converse also hold, i.e., if the counting property holds
%(and it is enough to require it for the single graph $F=C_4$) then the graphs must be $(p,o(n))$-jumbled. 
%For $k\geq 3$, however, this is not the case and a $(p,\beta)$-jumbled $k$-graph may not contain a single copy of, say, a $(k+1)$-clique, l{et al.}one the expected number of such copies.

\medskip

With regards to spanning subgraphs in quasi-random $k$-graphs (with a mild minimum vertex degree), the situation also turned out to be more complex. Indeed, for graphs the famous \
Blow-up Lemma~\cite{komlos1997blow} of Koml\'os, S\'ark\"ozy and Szemer\'edi implies that a quasi-random graph with linear minimum  degree contains any bounded-degree spanning subgraph. For $k$-graphs, no such universal statement is known but some natural spanning linear  subgraphs have been studied. 
Lenz and Mubayi~\cite{LenzMubayi1} and Lenz, Mubayi, and Mycroft~\cite{LMM} investigated the existence of 
perfect matchings,~$F$-factors for linear~$F$, and  loose Hamilton cycles.
As usual an \emph{$F$-factor} in a $k$-graph $H$ is a collection of vertex disjoint copies of $F$ in $H$ which {cover} all of $V(H)$. 
Furthermore,
a $k$-uniform \emph{loose cycle} is a $k$-graph whose 
vertices can be cyclically ordered in such a way that each of its edges consists of~$k$ consecutive vertices, and  each edge intersects the subsequent edge 
(where the edge ordering is inherited by the {ordering} of the vertices) in exactly one vertex.
We say that the $k$-graph $H$ contains a \emph{loose Hamilton cycle} if it contains a loose cycle on $|V(H)|$ vertices
 as a 
subgraph.
%We note that the results in~\cite{LenzMubayi1, LMM} can also be seen as the meeting point of quasi-randomness with another very popular and active line of research in extremal hypergraph theory, 
%which investigates the existence of spanning subgraphs in $k$-graphs with high minimum degree (but no quasi-randomness assumptions). 
%We refer to~\cite{} for some results in this direction and to~\cite{} for a survey.\tred{Do we want to mention some results here?}

For a $k$-graph $H$ and sets $U_1,\dots, U_{k-1} {\subseteq} V(H)$, let $\deg(v;U_1,\dots, U_{k-1})=e(\{v\},U_1,\dots, U_{k-1})$ denote the \emph{degree of $v\in V(H)$ in $(U_1,\dots, U_{k-1})$}.
When $U_i=U$ for all $i\in[k-1]$ let\footnote{Recall that we define $e(A_1,\ldots, A_k)$ to count labelled edges and so here any edge containing $v$ and vertices of $U$ is counted $(k-1)!$ times in $\deg(v;U)$.}   $\deg(v;U)=\deg(v;U_1,\dots, U_{k-1})$ and define  the  \emph{minimum vertex degree} of $H$ by
 $\delta(H)=\min_{v\in V(H)}\deg(v;V(H))$.
 %, i.e, the minimum $m$ such that each vertex of~$H$ is contained in at least $m$ edges.
Lenz and Mubayi showed the following concerning  factors in quasi-random $k$-graphs .
%Their results establish that for any $p\in(0,1)$ constant
% $\big(p,o(n^{k/2})\big)$-jumbled $k$-graphs with a mild minimum degree and natural divisibility conditions contain loose Hamilton cycles and $F$-factors, 
\begin{theorem}[\cite{LenzMubayi1}]\label{thm:densecaseF}
For all  $k\geq 2$,  $0<c,p<1$ and all linear $k$-graphs $F$ on $v_F$ vertices there exists an $n_0$ and an $\eps>0$ such that the following holds.
If  $H$ is a $(p,\eps n^{k/2})$-jumbled  $k$-graph on $n\in v_F\mathbb N$ vertices such that $n>n_0$, and $H$ has  minimum vertex degree $\delta(H)>c n^{k-1}$,
then $H$ contains an $F$-factor. \qed
%Furthermore,  if $n>n_0$ and $(k-1)|n$, then $H$ contains a loose Hamilton cycle.\qed
\end{theorem}
The result of Lenz, Mubayi and Mycroft concerning loose Hamilton cycles reads as follows.
\begin{theorem}[\cite{LMM}]\label{thm:densecaseHC}
For all  $k\geq 2$ and  $0<c,p<1$,  there exists an $n_0$ and an $\eps>0$ such that the following holds.
If  $H$ is a $(p,\eps n^{k/2})$-jumbled  $k$-graph on $n\in (k-1)\mathbb N$ vertices such that $n>n_0$, and  $H$ has  minimum vertex degree $\delta(H)>c n^{k-1}$,  
then $H$ contains  a loose Hamilton cycle.\qed
\end{theorem}

We note here that the results in~\cite{LenzMubayi1} and~\cite{LMM} are actually slightly stronger than stated and  apply to $k$-graphs satisfying only the lower bound of the edge count in \eqref{eq:defkjumbled}.

\subsubsection*{Sparse pseudo-random hypergraphs}
In the sparse regime (aka the pseudo-random case) the notion~\eqref{eq:defkjumbled} has been studied by Haviland and Thomason~\cite{HavTho} and later by Friedman~\cite{Friedman} and Friedman, Wigderson~\cite{FriedmanWigderson} in the context of hypergraph spectral gap.
Concerning the emergence of spanning subgraphs in pseudo-random $k$-graphs 
%Here, as in the graph case, we are interested in conditions on the parameters which guarantee the existence of a given spanning structure. 
we note that  there is a wealth of literature 
on this topic for graphs, i.e., when $k=2$ (see, e.g.,~\cite{kHC,sparseblowup,HKMP18,han2019finding,KrivelevichSudakov,krivelevich2004triangle,Nenadov}). 
%The purpose of the current paper is to explore corresponding questions in the $k$-graph setting. 
%We note here that as part of their paper~\cite{LenzMubayi1} Lenz and Mubayi also studied the emergence of perfect matchings in sparse pseudo-random $3$-graphs.%
%Our result will improve upon theirs in many aspects, see the discussion after Corollary~\ref{cor:matchings}.
For higher uniformity, Krivelevich, Frieze and Loh~\cite{FKL} used a stronger and more complicated notion of pseudo-randomness 
to study  packing of tight Hamilton cycles.
The only result concerning~\eqref{eq:defkjumbled} we are aware of  is due to Lenz and Mubayi~\cite{LenzMubayi1} who studied perfect matchings in sparse pseudo-random $3$-graphs.  
%With respect to~\eqref{eq:defkjumbled}, however, we are not aware of any work except the following concerning perfect matchings in sparse pseudo-random $3$-graphs by  Lenz and Mubayi~\cite{LenzMubayi1}. 
Their result is stated in terms of {hypergraph} eigenvalues, a notion originated in the work of Friedman and Wigderson~\cite{Friedman, FriedmanWigderson} for regular  $k$-graphs and extended 
to all  $k$-graphs by Lenz and Mubayi in~\cite{LenzMubayi_eig}.
%Here we only give the basic definitions concerning regular hypergraphs, referring the reader to~\cite{LenzMubayi_eig} for the more general notion of $\lambda(H)$. 
Given a $k$-graph $H=(V,E)$  (possibly with loops) on the vertex set $[n]=\{1,\dots,n\}$ and  the standard basis 
$\vece_1,\dots,\vece_n$ of $W=\RR^n$,  define the $k$-linear form $\tau_H:W^k\to\mathbb R$ by
\[\tau_H(\vece_{i_1},\vece_{i_2},\dots,\vece_{i_k})=\begin{cases}1&\text{if } \{i_1,\dots ,i_k\}\in E,\\ 0&\text{otherwise.}\end{cases}\]
By multi-linearity this thus determines $\tau_H$.  Let $J$ denote the all-one $k$-linear form, i.e., the form defined as above, but with 
$J(\vece_{i_1},\dots,\vece_{i_k})=1$ for all $i_1,\dots ,i_k\in[n]$.
% and we define $\tau_H$ to be \emph{$\Delta$-regular} if the form
%\[\sigma=\sigma_H=\tau_H-\frac \Delta nJ\]
%satisfies $\sigma_H(\vecone,\vecu_2,\dots,\vecu_k)=\sigma_H(\vecu_1,\vecone,\dots,\vecu_k)=\dots=\sigma_H(\vecu_1,\dots,\vecu_{k-1},\vecone)=0$ for all $\vecu_1,\dots,\vecu_k$ and with~$\vecone$
%being the everywhere one vector.
Let $\|\cdot\|$ denote the Euclidean 2-norm on $\RR^n$ and for a $k$-linear form $\phi\colon W^k\to\RR$ let its spectral norm be defined as
\[\|\phi\|=\sup_{\|\vecu_1\|=\dots=\|\vecu_k\|=1}|\phi(\vecu_1,\dots,\vecu_k)|\] Then the \emph{first} and the \emph{second eigenvalue} of $H$ are defined as
\[\lambda_1(H)=\|\tau_H\| \qquad \text{and}\qquad \lambda(H)=\left\|\tau_H-\frac{k!e(H)}{n^k}J\right\|.\]
%\sup_{\|\vecu_1\|=\dots=\|\vecu_k\|=1}|\sigma_H(\vecu_1,\dots,\vecu_k)|.\]
It was shown that $k$-graphs (\cite{FriedmanWigderson} for regular and~\cite{LenzMubayi_eig} for all $H$) with second eigenvalue $\lambda=\lambda(H)$
are $( p,\lambda)$-jumbled with $p=\frac{k! e(H)}{n^k}$.
% Friedman and Wigderson~\cite{Friedman, FriedmanWigderson} (and Lenz-Mubayi~\cite{LenzMubayi_eig} for all $H$) %($=\frac\Delta n$ when $H$ is $\Delta$-regular). % and thus $(\frac\Delta n,\alpha,\eps)$-pseudo-random, for any  $\eps>0$ and $\alpha=\frac{\lambda_2^2}{\eps^2p^2 n^k}$.
%By $(n,\Delta,\lambda)$-$k$-graph we mean a $\Delta$-co-regular $k$-graph with $\lambda_2(H)\leq \lambda$.
%Then we have the following for jumbled $k$-graphs and in particular $k$-graphs with small second eigenvalue. 
%We refer to these two works for further information and note here only that these works
%define a notion of \emph{second eigenvalue} for $k$-graphs and relates it
%to jumbledness
%through a Hypergraph Expander Mixing Lemma, showing that any $k$-graph $H$ with edge density $p= \frac{k! |E(H)|}{n^k}$, and second eigenvalue 
%\mbox{$\lambda(H)\leq \lambda$} is $(p,\lambda)$-jumbled. 
The following is  the  result by Lenz and Mubayi, which ours  will improve upon in several  aspects, see Theorem~\ref{cor:matchings} and the discussion thereafter. 

\begin{theorem}[\cite{LenzMubayi1}]\label{thm:sparsecasematching1}
For all  $0<c<1$ there exists an $n_0$ and an $\eps>0$ such that the following holds for all $n<n_0$.
Let  $H$ be a  $3$-graph on $n\in 3\mathbb N$ vertices, with edge density $p:=\frac{6|E(H)|}{n^3}$,  
minimum co-degree\footnote{That is, every pair of vertices in $V(H)$ is contained in at least $\delta_2(H)$ edges.} 
$\delta_2(H)>c pn^{2}$ and second eigenvalue \[\lambda(H)\leq \eps p^{16}n^{3/2}.\]
Then $H$ contains a perfect matching. \qed
%Furthermore,  if $n>n_0$ and $(k-1)|n$, then $H$ contains a loose Hamilton cycle.\qed
\end{theorem}

%We now turn our focus to the sparse regime
%and address the problem of emergence of spanning subgraphs in sparse sufficiently pseudo-random $k$-graphs.
%With the exception of  the above mentioned result of Lenz--Mubayi concerning perfect matchings, which our result will improve upon, we are not aware of any
%further work in this direction. 
\subsection*{Main results} Our results {rely} on the  following  weaker version of~\eqref{eq:defkjumbled}. %which  will be sufficient for our purposes.
Given a $k$-graph $H$ and $p,\alpha,\eps\in [0,1]$ we say that 
 $H$ is  \emph{$(p,\alpha,\eps)$-pseudo-random} if for all (not necessarily disjoint) 
 subsets $A_1,\dots, A_k{\subseteq} V(H)$  with
$|A_1|\cdots |A_k|\geq \alpha |V|^k$ we have
\begin{equation}\label{eq:defpseudo}
e(A_1,\dots,A_k)=(1\pm\eps)p\cdot|A_1|\cdots |A_k|. %\pm \beta p\cdot \vol(V_1,\dots,V_k),
\end{equation}
Our reasons to use this notion of pseudo-randomness are twofold. Firstly, it  is weaker than~\eqref{eq:defkjumbled}  and the notion of spectral gap studied by Friedman and Wigderson~\cite{Friedman,FriedmanWigderson} and Lenz and Mubayi~\cite{LenzMubayi_eig}. Indeed, 
note that for any $\eps>0$,  
a $(p,\beta)$-jumbled $k$-graph (and hence $k$-graphs  with density $p$ and second eigenvalue $\lambda(H)\leq \beta$) is 
$(p,\alpha,\eps)$-pseudo-random with $\alpha= \frac{\beta^2} {\eps^2 p^2 n^k}$. 
Secondly, we believe that this notion is the most natural for this work and free from unnecessary intricacies.  Indeed, 
in what follows  it will be sufficient to have $\eps>0$, which controls the variation from the average density, to be a sufficiently small constant. The parameter $\alpha=\alpha(n,p)$ then controls the  size of the vertex sets for which the condition is non-trivial. % we can control the edge distribution between. % is then determined by 
% choosing a suitable . 
When investigating subgraphs of  pseudo-random (hyper-)graphs, this is often the parameter which is at the crux of the proofs, forcing the degree of pseudo-randomness necessary.
  %This is often precisely the parameter we want to investigate in relation to  the appearance of subgraphs.
%This definition thus captures exactly the pseudo-random condition that we need for our purposes and does not impose the stronger control coming from~\eqref{eq:defkjumbled}.
%as we do not require a strong  control of the deviation from the expected value $p\cdot|A_1|\cdots |A_k|$.

Regarding the emergence of a perfect matching in sufficiently pseudo-random hypergraphs we show the following.
\begin{theorem}\label{cor:matchings}
For all integers $k\geq 3$ and $c>0$ there is an $\eps>0$ and an $n_0$ such that for any $n>n_0${,}
any $(p,\eps p^k,\eps)$-pseudo-random $k$-graph on $n\in k\mathbb N$ vertices with $\delta(H)\geq c pn^{k-1}$ contains a perfect matching. 

In particular, this holds if $H$ satisfies the minimum degree condition and is $(p, \beta)$-jumbled or is of density $p$ and has second eigenvalue 
$\lambda(H)\leq \beta$ with $\beta<\eps p^{k/2+1}n^{k/2}$. 
\end{theorem}
%Concerning the particular case of perfect matchings, 
Theorem~\ref{cor:matchings} is a direct consequence of Theorem~\ref{thm:Ftilings} from below which addresses the existence of $F$-factors for a linear $F$.
Theorem~\ref{cor:matchings} improves upon Theorem~\ref{thm:sparsecasematching1} for $3$-graphs, and extends it   
 to $k$-graphs. 
 Indeed, Theorem~\ref{thm:sparsecasematching1} relies on the stronger notion of pseudo-randomness
coming from the second eigenvalue of hypergraphs and requires that $\lambda =o(p^{16}n^{3/2})$, whilst ours only requires  $\lambda(H)=o(p^{5/2}n^{3/2})$. %our notion their condition implies that $H$ is $(p,o(p^{30}))$-jumbled. 
Moreover,  our result only relies on minimum vertex degree, whilst  Theorem~\ref{thm:sparsecasematching1} requires minimum co-degree condition, concretely, that all pairs of vertices in $H$ are contained in $\Omega(pn)$ edges, 
which  is a rather strong restriction in general.
Indeed, while  pseudo-randomness implies that most vertices have high degree, making    
the minimum vertex degree the natural one to consider in this context,
 it is easy to construct pseudo-random hypergraphs 
 with a substantial proportion of pairs of vertices having co-degree zero.  

Our approach also covers the case of loose Hamilton cycles 
and thus extends Theorem~\ref{thm:densecaseHC} to sparse pseudo-random $k$-graphs as follows.
%which is probably the first and most natural spanning connected linear $k$-graphs.
%Our approach can also accommodate the case of loose Hamilton cycle and yields the following.
\begin{theorem}\label{thm:hamcyc}
For any given integer $k\geq 3$ and  $c>0$ there is an $\eps>0$ and an $n_0$ such that for every $n>n_0$ the following holds.
Suppose that 
$H$ is a  $(p,\eps p^{k-1},\eps)$-pseudo-random $k$-graph on $n\in (k-1)\mathbb N$  vertices with $\delta(H)\geq c pn^{k-1}$.
%and   $\Delta_2(H)\leq Cp^{1-k}n^{k-2}.$ %,\frac{\eps}{\log n} p n^{k-1}$.
Then $H$ contains a loose Hamilton cycle.
%Moreover, for $k=3$ it suffices that $H$ is $(p,\eps p^{2},\frac{\eps}{\log (1/p)})$-pseudo-random.

In particular, this holds if $H$ satisfies the minimum degree condition and is $(p, \beta)$-jumbled or is of density $p$ and has second eigenvalue { $\lambda(H) \leq \beta$ with $\beta<\eps p^{(k+1)/2}n^{k/2}$. }
\end{theorem}

Finally we generalize Theorem~\ref{cor:matchings} and address the appearance of $F$-factors  in sufficiently pseudo-random $k$-graphs for any constant size linear $k$-graph $F$.
%With regards to subgraphs of  pseudo-random $k$-graphs, 
It should be noted that even the appearance  of a single copy of $F$ has not been explicitly studied for the sparse 
  case, neither for~\eqref{eq:defkjumbled} nor for~\eqref{eq:defpseudo}. Nevertheless, we  will show 
that the argument by Kohayakawa {et al.}~\cite{KNRS} for (dense) quasi-random $k$-graphs can be extended to the sparse range, i.e., 
that sufficiently pseudo-random $k$-graphs contain the ``expected'' number of copies of any constant size, linear $k$-graph~$F$. 
%, provided $\alpha$ is sufficiently small.  

But which degree of pseudo-randomness is sufficient for a given linear $k$-graph $F$? 
Which may be the crucial parameter(s) of $F$ which determine(s)/affect(s) the pseudo-randomness needed?
Even in the case of graphs the former question is notoriously difficult  and wide open, being  resolved for a handful of specific graphs $F$ only. 
While upper bounds  on the pseudo-randomness required do exist and are believed to be tight in many cases, 
matching lower bound constructions are rare and difficult.
A common
strategy to find a copy of $F$ in pseudo-random graphs relies on sequential 
embedding of its vertices, see e.g.~\cite{KRS, KRSSS, sparseblowup}. Hence
the (vertex) degeneracy of~$F$, i.e.,  $d_F=\max \{\delta (F')\colon F'\subset F\}$, the largest minimum degree over all induced subgraphs of~$F$, %
is usually a crucial parameter in determining the pseudo-randomness required.
We refer to~\cite{KRS, KRSSS, sparseblowup} for some results in similar settings in graphs which involve the vertex degeneracy (or its trivial 
upper bound, the maximum degree) or related notions of degeneracies~\cite{CFZ,ABSS}.

Unfortunately, vertex embedding strategies  fail when dealing with higher uniformities and this is probably one of the main reasons why 
the notion of linear quasi-randomness remained 
obscure for a long time. 
As it turns out, sequential edge embedding is a suitable strategy in this setting, making   
the following notion of edge degeneracy a natural parameter to consider.  

Briefly speaking the \emph{edge degeneracy} of a linear $k$-graph~$F$ is simply the vertex degeneracy of its line graph, i.e.,
the graph  on the vertex set $E(F)$ in which two distinct vertices $e,e'\in E(F)$ are connected if they intersect.
Let us expand on this notion!
For an edge $e$ in~$F$  we denote by  $\deg(e)$ its degree in the line graph of $F$.
Thus $\deg(e)=\sum_{v\in e} (\deg_{F}(v)-1)$  which is 
% $\deg(e)=\sum_{v\in e} (\deg_{F}(v)-1)$ and  
equal to the
number of edges $e'\neq e$ in~$F$ which intersect~$e$.
%, where $F-e$ denotes the $k$-graph $F$ with the edge $e$ removed.
%the sum of its vertex degree in $F-e$, the $k$-graph 
%In other words,  $\deg(e)=\sum_{v\in e} \deg_{F-e}(v)=\sum_{v\in e} (\deg_{F}(v)-1)$ is simply . 
%As $F$ is a linear $k$-graph we have $\deg(e)=\sum_{v\in e} \deg_{F-e}(v)=\sum_{v\in e} \deg_{F}(v)-k$, where $F-e$ denotes
%the $k$-graph $F$ with the edge $e$ removed.
Let $\delta' (F)=\min_{e\in F}  \deg(e)$  and $\Delta'(F)=\max_{e\in F}  \deg(e)$ denote the minimum and the maximum edge-degree of~$F$, respectively.  Similar to the vertex degeneracy, the {edge degeneracy} of~$F$ can be defined as
 \[\degen(F)=\max\left\{\delta'(F')\colon F'\subset F, V(F')=V(F)\right\},\]
the largest minimum edge-degree taken over all subgraphs $F'\subset F$ on the vertex set $V(F)$.
There are  further similarities between this notion of degeneracy and the vertex degeneracy.
%This parameter have many similarities with the vertex degeneracy. 
For example,  it is easily seen that $\degen(F)$ is the minimum~$d$ for which there is an ordering $e_1,\dots, e_s$ of 
the edges of~$F$~($s=|E(F)|$) such that each $e_i$ has  edge-degree at most $d$ in the spanning subgraph of $F$ induced by 
the edges $e_1,\dots, e_{i}$.
Moreover, the edge degeneracy is at most the maximum edge-degree, $\degen(F)\leq \Delta'(F)$.

%An \emph{edge exposure of $F$} is a permutation 
 %$\sigma\in S_s$ of the edges $E(F)=\{e_1,\dots, e_s\}$ and for {such a} $\sigma$ 
%and an $i\in[s]$  define the \emph{weight of the edge $e_{\sigma(i)}$} as the number of edges which intersect $e_{\sigma(i)}$ and which appear before $e_{\sigma(i)}$
%in the order given by $\sigma$, i.e.,
 %\[w_{\sigma(i)}= |\{j\colon \sigma(j)<\sigma(i)\text{ and }e_{\sigma(j)}\cap e_{\sigma(i)}\neq\emptyset\}|.\]
%The best exposure  will be the one which minimises the maximum edge weight. This minimum we call the \emph{edge degeneracy} of $F$ 
%\[\degen(F):=\min_{\sigma}\max_{i\in[s]} w_{\sigma(i)}  ,\]
%where the minimum is taken over all edge exposures of $F$. 
%As a concrete example, taking $F$ to be the $3$-uniform Fano plane, i.e. the unique finite projective plane of order 2, one can see that $\degen(F)=6$. Indeed, adding the edges in any order gives an exposure with maximum edge weight $6$ whilst noting that the last edge in any exposure has weight 6 shows that this is optimal. 

As a consequence of {one of our auxiliary results, Lemma \ref{lem:rootcount},} we  obtain the following concerning the appearance of a single copy of a fixed size linear $k$-graph. 
\begin{theorem} \label{prop:small subgraph appearence}
 For any linear $k$-graph $F$, there exists an $\eps>0$ and an $n_0$ such that for every $n>n_0$, any $n$-vertex $\big(p,\eps p^{\degen(F)},\eps \big)$-pseudo-random $k$-graph $H$ contains a copy of  $F$. 
 
 In particular, the same conclusion holds if $H$ is $(p, \beta)$-jumbled or is of density $p$ and has second eigenvalue 
 $\lambda(H)\leq \beta$ with $\beta<\eps p^{\degen(F)/2+1}n^{k/2}$. 
\end{theorem}

%To obtain an $F$-factor, however, we need a stronger pseudo-randomness condition. together with a mild condition on
%For our purposes, however, we will need a counting for a slightly  stronger notion of copies, which has certain ``root''  vertices of $F$ pre-embedded. 
%This requires a slightly stronger notion of edge degeneracy and
%Lemma~\ref{lem:rootcount} is the corresponding lemma, 
%which is the building block in the construction of the $F$-factors and loose Hamilton cycles.
%We write $\degen(F)$ only if $\cX$ is clear from context.
%Using counting of constant size linear $k$-graphs as the starting point 
%Our main result  concerning factors is the following.
We will in fact show that one has  roughly the correct number of copies of $F$ under the same pseudo-random conditions. 
The pseudo-randomness condition in Theorem~\ref{prop:small subgraph appearence} is known to be tight for graph triangles and conjectured to be tight for graph cliques and therefore cannot be improved in general. 
Indeed, for graph triangles we have $\degen(K_3)=2$  and 
by Alon's construction~\cite{A94} there are $K_3$-free {$n$-vertex $d$-regular graphs with $d=\Omega(n^{2/3})$, which are $(d/n,\lambda)$-jumbled with  $\lambda=O(n^{1/3})$.}

Moving on to the appearance of $F$-factors in sufficiently pseudo-random $k$-graphs {our} result  establishes the following. % for the existence of $F$-factors. % in a pseudo-random hypergraph. %, which also depends on the edge degeneracy of $F$. 
\begin{theorem}\label{thm:Ftilings}
For given integers $f\geq k\geq 2$ and  $c>0$ there is an $\eps>0$ and an $n_0$ such that for every $n>n_0$ the following holds.
Let $F$ be  a linear $k$-graph on $f$ vertices and let   \[\ell:=\degen(F)+\Delta'(F)+k.\]%\leq 2\Delta'(F)+k.\] %\max_{e\in E(F)}\sum_{v\in e}\deg_F(v).\]  
Suppose 
$H$ is a  $(p,\eps p^{\ell},\eps)$-pseudo-random $k$-graph  on $n\in f\mathbb N$ vertices with $\delta(H)\geq c pn^{k-1}$. 
%and   $\Delta_2(H)\leq Cp^{1-\ell}n^{k-2}.$ %,\frac{\eps}{\log n} p n^{k-1}$.
Then there is an  $F$-factor of $H$. 

In particular, this holds if $H$ satisfies the minimum degree condition and is $(p, \beta)$-jumbled or is of density $p$ and has second eigenvalue { $\lambda(H)\leq \beta$ with $\beta<\eps p^{\ell/2+1}n^{k/2}$. }
\end{theorem}
%As $\degen(F)\leq \Delta'(F)$,  we have $\ell\leq 2\Delta'(F)+k$  which may serve as
%a simple bound.
% for the parameter $\ell$ in the theorem is given 
%by $\ell\leq k + 2\max_{e\in E(F)}\deg(e)$.

%As a corollary of the theorem we obtain the following for the existence of perfect matchings.
%\begin{corollary}\label{cor:matchings}
%For all integers $k\geq 3$ and $c>0$ there is an $\eps>0$ and an $n_0$ such that for any $n>n_0${,}
%any $(p,\eps p^k,\eps)$-pseudo-random $k$-graph on $n\in k\mathbb N$ vertices with $\delta(H)\geq c pn^{k-1}$ contains a perfect matching. 

%In particular, this holds if $H$ satisfies the minimum degree condition and is $(p, \beta)$-jumbled or is of density $p$ and has second eigenvalue $\lambda(H)<\eps p^{k/2+1}n^{k/2}$. \qed
%\end{corollary}
Note that in $(p,\eps p^{\ell},\eps)$-pseudo-random $k$-graphs, with $\ell=\degen(F)$ we obtain a copy of $F$ by Theorem~\ref{prop:small subgraph appearence}, and with $\ell=\degen(F)+\Delta'(F)+k$ we get an $F$-factor by Theorem~\ref{thm:Ftilings}.
Recall that $\degen(F)\leq \Delta'(F)$,  thus $\ell\leq 2\Delta'(F)+k$  which may serve as
a simple upper bound on $\ell$.
By taking~$F$ to be a single edge we have $\Delta'(F)=0$, thus Theorem~\ref{thm:Ftilings} implies Theorem~\ref{cor:matchings} concerning perfect matchings. 
In general, Theorem~\ref{thm:Ftilings} extends Theorem~\ref{thm:densecaseF} from dense quasi-random $k$-graphs to sparse pseudo-random $k$-graphs and establishes the first explicit  condition  on  pseudo-randomness that guarantees the existence of a  factor.

%Indeed, 
It is difficult to comment on the tightness of our general results. 
Indeed, at this stage it is not even clear what to conjecture as the threshold for the appearance of subgraphs of pseudo-random hypergraphs. 
For graphs it is  believed that  the obstruction for finding an $F$-factor in a pseudo-random graph is in fact the appearance of a single copy of $F$. 
That is, essentially the same pseudo-random conditions that guarantee a single copy of $F$ actually guarantee an $F$-factor. 
This intuition is confirmed only for the case $F=K_3$  where there still remains a $\log$ factor between the upper and lower bounds~\cite{A94,Nenadov} and in general 
finding $F$-free pseudo-random graphs is a very challenging problem, see e.g.~\cite{bishnoi2020construction}. 

One may conjecture that the same phenomenon occurs in hypergraphs, i.e., that the threshold for the appearance of a single copy of a linear $k$-graph $F$
and that for the appearance of $F$-factors coincide. In this respect Theorem \ref{prop:small subgraph appearence} provides conditions for the appearance of a fixed linear hypergraph and so may serve as a benchmark for future work. 

Any further results providing constructions or  conditions for finding general subgraphs, even for constant sized subgraphs of hypergraphs, 
would be very interesting. 
Given the generality of Theorem \ref{prop:small subgraph appearence}  and also Theorem \ref{thm:Ftilings}, 
it will not too hard to
%we believe that it is 
 %possible 
 to improve on the 
conditions for particular linear $k$-graphs $F$.%, yet  we believe it would  be difficult to improve  on these results in general.
%That said, we believe it would  be difficult to improve  on our results in general and in particular our results for perfect matchings and loose Hamilton cycles, 
%namely Theorem \ref{cor:matchings} and Theorem \ref{thm:hamcyc}. %

%While we cannot say  much about the tightness of Theorem~\ref{thm:Ftilings}, Theorem~\ref{thm:hamcyc} and Corollary~\ref{cor:matchings}, 
%as good constructions are rare even in the graph case, 
%we do believe that our results concerning perfect matchings and loose Hamilton cycles  will not be easy to improve upon. 
%On the other hand, due to the generality of Theorem~\ref{thm:Ftilings} it is probably not too hard to find some particular linear $F$ for which the bound can be improved.

\subsection{Proof overview and organisation}\label{sec:outline}
%In the following section we lay out the  notation and conventions that we will use, give some auxiliary results concerning pseudo-random $k$-graphs and provide a proof outline. In the following section, Section~\ref{sec:counting}, we  then show the counting lemma for linear $k$-graphs $F$, Lemma~\ref{lem:rootcount}. 
%For our purposes we will be interested in copies with certain ``root''  vertices of $F$ pre-embedded.
% Lemma~\ref{lem:rootcount} then will be used to show  Lemma~\ref{lem:absorber} which can embed a linear size and flexible structure.
%Lemma~\ref{lem:absorber} will be used in various situations, in particular, to establish the absorbing lemma, Lemma~\ref{lem:keyproperty} from Section~\ref{sec:absorber}, 
%which allows us to extend an almost $F$-factor to an $F$-factor and an almost spanning loose cycle to a Hamilton cycle.
%The proofs of Theorem~\ref{thm:Ftilings} and Theorem~\ref{thm:hamcyc} will then be given in Section~\ref{sec:proofthms}.
As mentioned above, Theorem~\ref{cor:matchings} is a consequence of Theorem \ref{thm:Ftilings}. The proofs of Theorem \ref{thm:Ftilings} and Theorem \ref{thm:hamcyc} work by \emph{absorption}, {a method popularised by R\"odl, Ruci\'nski and Szemer\'edi, see e.g. \cite{rodl2006dirac}. Both proofs follow the same scheme and so
 we will deal with them together.  In} the following we give a brief outline, ignoring some technical details. 

The main step is to show that  a sufficiently pseudo-random $H=(V,E)$ 
contains an \emph{absorbing} set $A\subset V$ as follows: there is a \emph{flexible} set $Z\subset A$ and an integer $m=\Omega(n)$
so that
\begin{description}
\item[$F$-factors] for \emph{any} $Z'\subset Z$ of size $m$
 the induced $k$-graph $H[A\setminus Z']$  contains an $F$-factor, 
 \item[Ham-cyc] for \emph{any} $Z'\subset Z$ of size $m$ the induced $k$-graph  $H[A\setminus Z']$  contains  
a spanning loose  path with some fixed end vertices $a_1$ and $a_2$ independent of~$Z'$.
\end{description}
{Considering}  the remaining vertices $V\setminus A$ and with the flexibility in mind,  we find some $Z'\subset Z$ of size $m$ so that 
 $H\big[(V\setminus A)\cup Z'\big]$ can be covered with disjoint $F$-copies or a spanning loose path with end vertices $a_1$ and~$a_2$, respectively.
%The vertices $V(\cA)$ which are involved in the absorbing structure  will be a small linear (in terms of $|V(H)|$) set. 
%We will then show that almost all of the vertices of $V(H)\setminus A $ can be covered by disjoint copies of $F$ or by a loose path~$P$, respectively, 
%leaving some small  set of leftover vertices $L\subset V(H)\setminus A$. 
%Using some subset $Z'\subset Z$ of size $m$, we cover the vertices in $L$ by copies of $F$ in the case of $F$-factors.
%In the case of Hamilton cycles we use some $Z'\subset Z$ of size $m$ to extend $P$ to a loose path which covers $L$ and has end vertices $a_1$ and~$a_2$. 
The flexibility property of $Z$ then implies that $H$ contains
 an $F$-factor or a loose Hamilton cycle, respectively. 

The absorbing set $A$ is obtained as the vertex set  of an
 absorbing structure $\cA$ in $H$, which is a family of  copies of  special $k$-graphs called \emph{absorbers}. These absorbers and
  their properties can be found in Section~\ref{sec:absorber}, see Lemma~\ref{lem:AF} for $F$-factors and Lemma~\ref{lem:AHC3} for loose Hamilton cycles. 
To ensure the absorbing property {of} $A=V(\cA)${,}  
the copies in~$\cA$ are not disjoint but overlap according to a certain prescribed structure called {a} \emph{template} (see~Lemma~\ref{lem:template}), a concept introduced by Montgomery~\cite{M14a,montgomery2019spanning}. 
This  approach requires that we deal with \emph{rooted} copies of  absorbers, i.e., copies in which the overlapped (aka root) vertices are pre-embedded.
In Lemma~\ref{lem:rootcount}  from Section~\ref{sec:embeddingsubgraphs} we show how to find such rooted copies in sufficiently pseudo-random hosts and in
Lemma~\ref{lem:absorber} we show how to put many of them together while controlling the intersection structure of the root vertices. 
When the intersection structure is a suitable  template 
this lemma yields the absorbing structure $\cA$, but it will also be useful in other steps of the proof.

In Section \ref{sec:proofthms} we prove Theorem \ref{thm:Ftilings} and Theorem \ref{thm:hamcyc} both together. 
We close the section with the following. %some notations and properties of pseudo-random $k$-graphs.

\subsection*{Notation and  properties of pseudo-random hypergraphs}

Throughout the paper we omit floor and ceiling signs where they do not affect the arguments.
Further, we write $\alpha\ll \beta\ll \gamma$ to mean that it is possible to choose the positive constants $\alpha, \beta,\gamma$ from right
to left. More precisely, there are increasing functions $f$ and $g$ such that, given $\gamma$, whenever we choose some
$\beta<f(\gamma)$ and $\alpha<g(\beta)$, the subsequent statements hold. Hierarchies of other lengths are defined similarly.

As previously mentioned, we write $x=y\pm z$ to denote that $y-z\leq x\leq y+z$. We also write equations such as $y_1\pm z_1=y_2\pm z_2$ which means that  $y_1+z_1\leq y_2+z_2$ and $y_1-z_1\geq y_2-z_2$. Thus if $x=y_1\pm z_1$ and $y_1\pm z_1=y_2\pm z_2$ then $x=y_2\pm z_2$.  More involved equations using $\pm$ should all be read the same; if we turn all the $\pm$ signs to $+$, then the equation holds if we replace all the $=$ signs  with  $\leq$ and if we  turn all the $\pm$ signs to $-$, the equations hold with $\geq$ signs replacing the $=$ signs. 
% We omit the use of floors and ceilings unless it is necessary, so as not to clutter the arguments.
%At times we choose constants $0\ll c_1 \ll c_2 \ll \ldots \ll c_k$, which means that one can choose constants from right to left so that all the subsequent constraints are satisfied. 
%That is, there exist increasing functions $f_i$ for $i\in [k]$ such that whenever $c_{i}\leq f_{i+1}(c_{i+1})$ for all $i\in[k-1]$, all constraints on these constants that are in the proof, are satisfied.   
%\subsection{Properties of pseudo-random hypergraphs}
Next, we collect some easy consequences of the definition of pseudo-randomness  \eqref{eq:defpseudo}. 
%We first show that if we take a large enough induced subgraph of a pseudo-random hypergraph, we can also estimate its edge distribution as it inherits a pseudo-random property from the host graph. We also show how to give a weaker estimate on the number of edges induced by smaller vertex subsets. 

\begin{fact}\label{fact:subset}
Given a  $(p,\alpha,\eps)$-pseudo-random hypergraph $H$ and $U,V_1, \dots, V_k{\subseteq} V(H)$;
\begin{itemize} \item For any $\gamma>0$, if  $|U| \geq \gamma^{\frac{1}{k}} |V|$, 
then $H[U]$ is $(p,\alpha/\gamma ,\eps)$-pseudo-random. 
\item If $|V_1|\cdots |V_k|<\alpha |V|^k$, then
\[\begin{aligned}\label{eq:smallUi}|e(V_1,\dots,V_k)-p\cdot|V_1|\cdots |V_k||\leq (1+\eps)p\cdot \alpha |V|^k.\end{aligned}\]
\end{itemize}

%Given a  $(p,\alpha,\eps)$-pseudo-random tuple $(V_1,\dots, V_k)$ is  and $V_1'\subset V_1, \dots, V_k'\subset V_k$. 
%\begin{itemize} \item If  $\vol(V_1',\dots, V_k')\geq \gamma \vol(V_1,\dots, V_k)$, 
%then $(V_1',\dots, V_k')$ is $(p,\alpha/\gamma ,\eps)$-pseudo-random. 
%\item If $\vol(V_1',\dots, V_k')<\alpha \vol(V_1,\dots, V_k)$, then
%\begin{align*}\label{eq:smallUi}|e(V_1',\dots,V_k')-p\cdot\vol(V_1',\dots,V'_k)|\leq (1+\eps)p\cdot \alpha \vol(V_1,\dots,V_k).\end{align*}
%\end{itemize}
\end{fact}
\begin{proof}
The first property follows directly from the definition. For the second
we extend each $V_i$ to a $U_i\subset V$ so that 
$|U_1|\cdots|U_k|=\alpha |V|^k$. Then pseudo-randomness yields
\[0\leq e(V_1,\dots, V_k)\leq e(U_1,\dots, U_k)\leq (1+\eps) p\cdot\alpha |V|^k,\]
and the second property follows when  $e(V_1,\dots, V_k)\geq  p\cdot|V_1|\cdots |V_k| $. When $p\cdot|V_1|\cdots |V_k|>  e(V_1,\dots, V_k)$ the second property is immediate. 
\end{proof}
Our next lemma shows that a pseudo-random hypergraph cannot be too sparse. 
\begin{lemma} \label{lem:lowerboundonp}
Let $\eps>0$ and $H$ be $(p,\eps p^\ell,\eps')$-pseudo-random $k$-graph such that $\eps'\leq 1/2$. Then $p=\Omega(n^{-s})$, where $s:=\frac{k(k-1)}{\ell(k-1)+k}$. In particular, if $\ell\geq k-1$, then $p = \omega(\log n/n)$. 
\end{lemma}
%\begin{proof}
%We have that $e(H)\leq pn^k+\eps' p n^k \leq 2pn^k$. Extending Tur\'an's theorem to $k$-graphs, Spencer \cite{spencer1971turan} (see also \cite[p.~434]{berge1973graphs}) showed that any $k$-graph with average  (vertex-)degree $d$ has an independent set of size $\Omega(nd^{-1/(k-1)})$. Hence $H$ has an independent  set $I$ of size  $\Omega\left(\frac{n}{(pn^{k-1})^{1/(k-1)}}\right)=\Omega\left(p^{-\frac{1}{k-1}}\right)$. Now clearly $e(I,\ldots,I)=0$. However, by pseudo-randomness, if $|I|^k\geq  \alpha n^k$, we have that $e(I,\ldots,I)\geq p|I|^k(1-\eps')>0$, a contradiction.
%
%Now if $p=o( n^s)$, then $p^{\ell+\frac{k}{k-1}} = o( n^{-k})$ and hence $|I|^k=\Omega(p^{-\frac{k}{k-1}})=\omega( p^\ell n^k )=\omega(\alpha n^k)$ and so from the previous paragraph we obtain a contradiction. Thus, $p=\Omega(n^s)$.
%\end{proof}
\begin{proof}
Extending Tur\'an's theorem to $k$-graphs, Spencer~\cite{spencer1971turan} (see also \cite[p.~434]{berge1973graphs}) 
showed that any $k$-graph with average  (vertex-)degree $d$ has an independent set of size $c nd^{-1/(k-1)}$ for some $c=c(k)>0$. 
As $e(H)\leq pn^k+\eps' p n^k \leq 2pn^k$ we infer that $H$ has an independent  set $I$ of size  $c{n}{(2kpn^{k-1})^{-1/(k-1)}}= c' p^{-\frac{1}{k-1}}$ for some $c'=c'(k)>0$. Clearly $e(I,\ldots,I)=0$, yet, if $|I|^k\geq  \eps p^\ell n^k$ then  pseudo-randomness implies  $e(I,\ldots,I)\geq p|I|^k(1-\eps')>0$, which is a contradiction.
Thus, $\eps p^\ell n^k>|I|^k\geq (c'p^{-\frac{1}{k-1}})^k$ which then yields $p=\Omega(n^{-s})$.
%Now if $p^{\ell+\frac{k}{k-1}} \le (c^k/\epsilon) n^{-k}$ and hence $|I|^k\ge c^k p^{-\frac{k}{k-1}}\ge \epsilon p^\ell n^k $ then from the previous paragraph we obtain a contradiction. Thus, $p=\Omega(n^s)$.
\end{proof}

\section{Finding small subgraphs in pseudo-random hypergraphs} \label{sec:embeddingsubgraphs}

A \emph{rooted $k$-graph} is a pair $(F,\cX)$ with a $k$-graph $F$ on a vertex set~$V(F)=\{x_1,\dots,x_r,u_1,\dots, u_f\}$ and the (possibly empty) tuple
$\cX=(x_1,\dots,x_r)$ of specified vertices such that
 for every two vertices $x_i, x_j$ of $\cX$   any  edge   containing $x_i$ is disjoint from any  edge   containing $x_j$. Vertices in $\cX$ are called  \emph{roots} of $(F,\cX)$ (or simply of~$F$).
Our aim is to  find ``rooted copies'' of $(F,\cX)$ in a %``root-compatible'' and 
``sufficiently'' pseudo-random~$H$. 

%for $i\in[t]$ the 
% $w(e_i)=w(e_i,\sigma)=\sum_{v\in e_i}\deg_{F_{i-1}}(v)$ where $F_0$ is the $F_{i-1}$
 %is  such that 
%\begin{itemize}
%\item in the ordering \item for  each $i=0,\dots,t$ the graph $F_i$ has the vertex set $V(F)$ and the edge set  $E(F_i)=\{e_1, \cdots, e_i\}$ where $E(F_0)=\emptyset$. 
%In particular, we have that $F_t=F$.
%\end{itemize}
%For each $i\in[t]$ we define    which we refer to as the \emph{weight of $e_i$ (w.r.t. $\sigma$)}.
Formally, let~$H$ be a $k$-graph with specified vertices $(y_1,\dots, y_r)=\cY$ and  $U\subseteq V(H)$ a vertex subset. %$\cV=(V_1,\dots, V_f)$ a tuple of not necessarily disjoint subsets of $V(H)$. 
%For a  set $g=\{u_{i_1},\dots,u_{i_{k-1}}\}\subseteq V(F)\setminus\{x_1,\dots,x_r\}$ we write
%$\cV(g)$ to denote $V_{i_1},\dots, V_{i_{k-1}}$ and  ${\deg(y_i; \cV(g))}$ to denote the number of edges containing $y_i$ and one vertex in each of the $V_{i_j}$.
A \emph{rooted copy} of $(F,\cX)$ in $(H,\cY,U)$ (or simply of $F$ in $H$) is an (edge preserving) embedding $\phi\colon V(F)\to V(H)$ such that $\phi(x_i)=y_i$ for all $i\in[r]$ and $\phi(u_i)\in U$ 
for all $i\in[f]$.

To deal with rooted $k$-graphs we need to extend our notion of edge degeneracy.
Recall that the edge degeneracy $\degen(F)$, for a linear $k$-graph $F$, is the vertex degeneracy of its line graph. 
Equivalently,  let $\deg(e)=\sum_{v\in e}(\deg(v)-1)$ be the degree of $e$ in the line graph of $F$, which is 
equal to the number of edges $e'\neq e$ in~$F$ which intersect~$e$. 
Then $\degen(F)$ is the minimum~$d$ for which there is an 
 \emph{edge exposure}, i.e.,  a permutation
$\sigma\in S_s$ of the edge set $E(F)=\{e_1,\dots, e_s\}$, such that each $e_{\sigma(i)}$ has  degree at most $d$ in the subgraph of $F$ on the vertex set $V(F)$ and with the edges $e_{\sigma(1)},\dots, e_{\sigma(i)}$. 
For a rooted $(F,\cX)$ we additionally require  from the edge exposure $\sigma$ from above that
all edges containing a root appear before edges not containing any root, i.e., there are no $i>j$ such that $e_{\sigma(i)}$ contains a root and 
$e_{\sigma(j)}$ does not. At times, we will write $\degen(F,\cX)$ to make it clear that we refer to the rooted graph and hence only consider these restricted edge exposures when calculating the degeneracy.  However we will also simply write $\degen(F)$ if $\cX$ is clear from context.

\subsection{A counting lemma for rooted $k$-graphs} \label{sec:counting}
When $H$ is sufficiently pseudo-random with respect to a fixed rooted graph $F$, and 
satisfies certain mild degree conditions, then the following lemma  guarantees many rooted copies of $F$ in $H$.
%with the specified vertices being images of the roots.
It is an extension of an argument by Kohayakawa et al.~\cite{KNRS} to the sparse case.  
%In general, the condition on the pseudo-randomness of $H$ is tight e.g. when $k=2$ and $F$ is $K_3$ (see \cite{A94}). However, for certain $F$, this can be improved e.g. for larger odd cycles in graphs (see \cite{KrivelevichSudakov}).

\begin{lemma}[Rooted counting]\label{lem:rootcount}
For  integers $k, f\geq 2$, $r\geq 0$ and  $1\geq c >0$ there is an $\eps>0$ and an $n_0\in N$ such that the following holds for all $n\geq n_0$.
Let~$\big(F,(x_1,\dots, x_r)\big)$ be a rooted linear $k$-graph  on $r+f$ vertices and with edge degeneracy $\ell$.
%given along with an edge exposure $\sigma$.
Suppose that~$H$ is a $(p,\eps p^\ell,\eps)$-pseudo-random $k$-graph on~$n$ vertices with\footnote{Here and throughout, $\Delta_2(H)$ denotes the maximum $2$-degree in $H$ i.e. $\Delta_2(H):=\max_{u\neq v\in V(H)}|\{e\in E(H):\{u,v\}\in e\}|$.}  $\Delta_2(H)<\eps pn^{k-1}$, $U\subseteq V(H)$ a set of size 
$|U|\geq c n$ and $y_1,\dots, y_r\in V(H)$ vertices, which
satisfy $\deg(y_i;U)\geq cp|U|^{k-1}$ for each $i\in[r]$.
Then there are at least 
\[\frac12(cp)^{e(F)} |U|^f\] rooted copies of $\big(F,(x_1,\dots,x_r)\big)$ in $\big(H,(y_1,\dots,y_r),U\big)$.%labeled copies $\phi$ of $F$ such that $\phi(x_i)=y_i$ for all $i\in[r]$ and  $\phi(u_i)\in V_i$ for all $i\in [f]$.

%In particular, if $F$ is an unrooted linear $k$-graph, $H$ a $(p,\eps p^{\degen F},\frac\eps{\log (1/p)})$-pseudo-random $k$-graph and
 % $U\subset V$ of size $|U|\geq \eps n$, then $H[U]$ contains a copy of $F$.
 % {The in particular part probably not needed}
\end{lemma}
Note that  $\Delta_2(H)<\eps pn^{k-1}$ is a rather weak condition, which moreover can be dropped if $\ell\geq k-1$. Indeed, in this case $p\gg1/n$   by \
Lemma~\ref{lem:lowerboundonp} and thus $\Delta_2(H)\leq n^{k-2}\ll pn^{k-1}$.

%$p=p(1)\gg\frac 1n$ can be dropped if $\ell\geq k-1$.

%$(p,\eps p^\ell,\eps)$-pseudo-random $H$ with $\ell\geq k-1$ immediately satisfy $p=p(n)\gg1/n$ by Lemma~\ref{lem:lowerboundonp}.

\begin{proof}
%The second part of the lemma follows from the first by taking $r=0$ and $V_i=U$ for all $i\in[f]$. {We need $p$ not too small also but maybe there is a general bound from pseudo-randomness}
%We now prove the first part of the lemma.
Given  $k, f,r$ and $c$ we choose $0<1/n_0 \ll \eps \ll \gamma\ll  {1/{f^{2}}}, c$. %$\eps=\frac{\log (1/1-\gamma)}{10f^3}$. %$\eps =\gamma^{k+1}/2^{f^2+k}$.
Fix $H$, $U$, and $y_1,\dots, y_r$ satisfying the assumptions of the lemma. 
Without loss of generality we assume that $cp|U|^{k-1}\leq\deg(y_i;U)\leq (1+\eps)cp|U|^{k-1}$ for each $i\in[r]$, by passing to a subgraph of $H$ if necessary. 
For a rooted linear $k$-graph $\big(F,(x_1,\dots, x_r)\big)$ 
let $t_{F}=e(F)-\sum_{i\in[r]}\deg_{F}(x_i)$ denote the number of edges containing no root vertices.
By induction on $t=t_{F}$ we show that for any such $\big(F,( x_1,\dots, x_r)\big)$ on at most $r+f$ vertices and with 
edge degeneracy at most $\ell$, there are

\begin{align}\label{eq:rootcount}
(1\pm{(t+1)}\cdot\gamma)c^{e( F)-t}p^{e( F)}|U|^{v( F)-r}
\end{align}
rooted copies of $\big( F,( x_1,\dots, x_r)\big)$ in $\big(H,(y_1,\dots,y_r),U\big)$.
%\vol(V_1,\dots,V_f)\] 
%rooted copies of $F$, where $t$ is the number of edges of $F$ which contain no root vertex. 
As $t_{ F}\leq \binom f2$ the lemma then follows by the choice of~$\gamma$. 
%We will proceed by induction on $t$.
%, which we define to be the number of edges of $F$ which contain no root vertex. 

%Suppose $F$, $\cX$, $\sigma$ and $H$, $\cY$, $\cV$  satisfy the assumptions of the lemma. % and by possibly relabelling the edges assume that $\sigma$ is the identity.  
%Recall that $F^*$ is the subgraph of $F$ obtained by removing all edges of $F$ not containing a root vertex and that by compatibility the
%number of rooted copies of $F^*$ in $H$ at least $(cp)^{e(F^*)}\vol(V_1,\dots,V_f)$. By possibly deleting edges containing the $y_i$'s 
%we assume that the number of rooted copies of $F^*$ is indeed $(1\pm\gamma)(cp)^{e(F^*)}\vol(V_1,\dots,V_f)$.
%This establishes the induction base, i.e., the case $t=0$ where $F=F^*$.
Consider first the case $t=0$, i.e., all edges of $F$ contain some root vertex. As
$F$ is linear and edges of $F$  containing different root vertices are disjoint, 
a rooted copy of $F$ is simply a disjoint union of stars with each star centered at some $y_i$. 
The 
 degree conditions   
%root-compatibility of $F$ and $H$ 
therefore yield the correct count on the number of rooted copies. 
Indeed, for any $y_i$ and any set $X\subset U$  of at most $r+f$ vertices  
 the number of edges containing $y_i$ and a vertex from $X$ is  
 at most $|X|\Delta_2(H)$.
 %trivially at most $(r+f)n^{k-2}$, which by $p\gg n^{-1}$ and $|U|\geq cn$ is at most  $\eps cp|U|^{k-1}$ for  large $n$.
%Indeed, the number of edges containing any fixed pair of vertices is trivially at most $n^{k-2}$ which, for sufficiently large $n$, is less than $\eps cp|U|^{k-1}$, due to Lemma \ref{lem:lowerboundonp}. Hence for any set of at most $f$ vertices $U'$, and any $y_i$, 
Thus, for each $i\in[r]$ and each of the  $\deg_F(x_i)$ edges in $F$ containing $x_i$ there are 
$\deg_H(y_i;U)\pm (r+f)\Delta_2(H)=(1\pm\eps)cp|U|^{k-1}\pm (r+f)\eps pn^{k-1}$
ways to choose the image of this edge
to build a copy of $F$. 
%there are at least $\deg(y_i;U)-f\eps cp|U|^{k-1}\geq (1-f\cdot\eps)cp|U|^{k-1}$ (ordered) edges containing $y_i$ and vertices of $U\setminus U'$. 
% implies therefor  that %for any $i\in[r]$ and any $e=\{x_i,u_{i_1,\dots,u_{i_{k-1}}}\in E(F)$  we have that
%the co-degree of $y_i$ and any $v_{i_j}$ into $\cV(e-x_i-u_{i_j})$ is at most  
With~$I$ denoting the number of isolated vertices of $F$,  the number of rooted copies of $F$ in~$H$ is then
\[(1\pm\gamma/2)\prod_{i\in[r]}\big(cp|U|^{k-1}\big)^{\deg_F(x_i)}\big(|U|\pm (r+f)\big)^{I}= (1\pm\gamma)(cp)^{e(F)}|U|^{v(F)-r},\] 
%Given our upper bound assumption on $\deg(y_i;U)$, a similar calculation gives that the number of rooted copies of $F$ in $H$ is within $(1\pm\gamma)(cp)^{e(F)}|U|^{f'}$. 
%Now we can delete edges of $H$ containing $y_i$'s so that the number of rooted copies of $F$ in $H$ is within $(1\pm\gamma)(cp)^{e(F)}|U|^f$. Let the resulting hypergraph be $H'$ which is where we 
proving the induction base. 
%By the assumptions on the degrees 
%$\deg(y_i;\cV(e-x_i))$ this number is at least $(cp)^{e(F)}\vol(V_1,\dots, V_f)$ and at most $((1+\eps)cp)^{e(F)}\vol(V_1,\dots, V_f)$. This establishes the induction base.
%Note that we have an upper bound on the number of copies which differs from the lower bound by a factor of at most $(1+2\eps)^{|E_1(F)|}$.
%\[\prod_{i\in[r]}\prod_{e\ni x_i}\deg(x_i;\cV(e-x_i))=\prod_{i\in[r]}\prod_{e\ni x_i}p(y_i,e-x_i)\vol(V_1,\dots,V_f).\]

%For the induction step assume now that $F$ has $t\geq 1$ root-free edges (and $f'$ non-root vertices) and assume 
%that the number of rooted copies in $(H,\cY,U)$ of any linear $k$-graph $(\tilde F,\cX)$  with $s<t$ root-free edges, $f'' \leq f$ non-root vertices and degeneracy at most $\ell$ is within 
%a compatible $\tilde H$  
%\[(1\pm(t-1)\cdot\r)c^{e(\tilde F)-s}p^{e(\tilde F)}
%|U|^{f''}.\] 
%\vol(V_1,\dots, V_f)$.
 %Consider the last edge $e_{\sigma(t)}$ in the exposure $\sigma$ and assume that $

For the induction step let $t\geq 1$, let
$F$ be a $k$-graph  with $t_F=t$, and let
%Next let  $t\geq 1$ and let $\big(F,(x_1,\dots,x_r)\big)$ be given  with at most $r+f$ vertices, 
%edge degeneracy at most $\ell$ and with $t=t_F$.
 $\sigma$ be an edge exposure which certifies the edge degeneracy of $F$. 
Note that the  induction hypothesis applies  to any proper subgraph of $F$ as it has edge degeneracy at most $\ell$
and strictly fewer edges.
% we can apply the induction hypothesis to any subgraph of $F$ with strictly fewer edges than $F$. 
%Indeed, by revealing the edges of the subgraph in the order they are revealed by the edge exposure $\sigma$, 
%one can see that the degeneracy of a subgraph of $F$ is at most that of $F$. 
 Let $F' = F_{t-1} = (V(F), E(F)\setminus \{e_{q}\})$ where $e_q$ 
is the last edge of $F$ according to the ordering  $\sigma$, i.e., $q:=\sigma\big(t+\sum_{i=1}^r \deg_F(x_i)\big)$. Note that $F'$ is defined on the same vertex set as $F$. For 
 a labelled copy $T$ of~$F'$ in $H$ we denote by $K_{T}$ the $k$-set of vertices of $T$ which corresponds to~$e_{q}$ in $F$. 
Let $\vecone_H: \binom V{k} \rightarrow \{0, 1\}$ be the indicator function of the edge set of $H$. In this notation a copy $T$ of $F'$ in $H$ extends to a copy of $F$ if and only if 
$\vecone_H(K_T)=1$,
consequently, summing over all copies $T$ of $F'$ in $H$  the number of copies of $F$ in~$H$ is
%\begin{align*}
\[
\sum_{T\subseteq H}\vecone_H(K_T) = \sum_{T\subseteq H} (p + \vecone_H(K_T) - p) =\sum_{T\subseteq H} p+\sum_{T\subseteq H} (\vecone_H(K_T) - p).
%&= p\cdot (1\pm\r/2)((1\pm \eps)cp)^{|E_1(F)|}p^{|E_2(F)|-1} \vol(V_1,\dots, V_{f}) + \sum_{L\subseteq H} (1_H(e_L) - p) \\
%&=  (1\pm\r/2)((1\pm \eps)cp)^{|E_1(F)|}p^{|E_2(F)|} \vol(V_1,\dots, V_{f}) + \sum_{L\subseteq H} (1_H(e_L) - p).
\]
%\end{align*}
%As $F'$ and $H$ are compatible with the same parameters
Noting that $e(F')=e(F)-1$ and $v(F')=v(F)$ the induction hypothesis yields 
\[\sum_{T\subseteq H} p= (1\pm t\cdot\r)c^{e(F)-t}p^{e(F)}|U|^{v(F)-r}\]
and in the following we will give a  bound to the error term $ \sum_{T\subseteq H} (\vecone_H(K_T) - p)$.

Without loss of generality suppose that $e_{q}=\{u_1,\dots,u_k\}$ and let $F_*=F[V(F)\setminus e_{q}]$ be the subhypergraph of $F$ obtained by removing the vertices  
$u_1,\dots,u_k$. 
Due to linearity  any edge in $E(F')\setminus E({F_*})$ intersects ~$e_{q}$ in at most one $u_i$. Hence,
for any copy $T_*$ of $F_*$ there are  sets $W_{i}\subseteq U$, $i\in[k]$, such that any $k$-tuple  $K\in W_1 \times\ldots \times W_k=:\ext(T_*)$ extends $T_*$ to a copy of $F'$. 
{Explicitly},~$W_i$ is the intersection of the neighbourhoods of the $(k-1)$-sets in $T_*$, which are the images of those $(k-1)$-sets in $F_*$ contained in an edge with $u_i$ in $F'$.
(Such a  copy of $F'$ then extends to a copy of $F$ if and only if $K\in E(H)$.) Let $z:=p^{w_q}|U|^k$, where $w_{q}$ is the \emph{weight}  of~$e_q$ according to the edge 
exposure $\sigma$, that is $w_q=\sum_{i=1}^k \deg_{F'} (u_i)$ . %, where $w_q$ is the weight of the edge $e_q$ according to the edge exposure $\sigma$.  
%where $w_t:=\sum_{i\in e_t} \deg_{F'} (v_i)$.  %then we ``expect''   $T_*$ to have $\ext(T_*)$ of size about .
% = \sum_{i\in e_t} \deg_F(v_i) - k\le m-k$ as $F$ is $m$-bounded and that $e(F)=e(F')+1 = e(F_*)+\ell+1$ as $F$ is linear. we ``expect''
 Using that $w_q\leq \ell$ and Fact~\ref{fact:subset}, we have that $H[U]$ is $(p,\eps' p^{w_q},\eps')$ pseudo-random for $\eps'=\sqrt{\eps}$ (using that $\eps\ll c$). There we obtain that

\begin{align}\label{eq:T*}\left|\sum_{T\subseteq H} (\vecone_H(K_T) - p) \right|&\le \sum_{T_*\subseteq H} \left| \sum_{K\in \ext(T_*)} (\vecone_H(K) - p) \right|=
 \sum_{T_*\subseteq H} \Big|e(W_1,\dots,W_k) - p|\ext(T_*)| \Big| \nonumber \\
&\le \sum_{\substack{T_*\subseteq H\\ |\ext(T_*)|\geq \eps' z}} \eps' p|\ext(T_*)|+\sum_{\substack{T_*\subseteq H\\|\ext(T_*)|<\eps' z}}(1+\eps')\eps' p\cdot z. \end{align}
Here, the estimate for the first sum comes from the definition~\eqref{eq:defpseudo} whilst the  second sum follows from~Fact~\ref{fact:subset}. % when $|\ext(T_*)|<\eps p^\ell |U|^k$ and 
%from the definition \eqref{eq:defpseudo} when $\eps p^\ell |U|^k\leq |\ext(T_*)|<\eps z$.
Note that  each edge in $E(F)\setminus (E(F_*)\cup\{e_{q}\})$ contains exactly one vertex from $e_{q}$, 
hence  we have   $e(F)=e(F')+1=e(F_*)+w_{q}+1$. Thus, with $t_*=t_{F_*}<t$ denoting the number of root-free edges of $F_*=(F_*,(x_1,\ldots,x_r))$,
 %Further, as $F_*\subset F$ induces a $\cX_*$, $\cY_*$, $\cV_*$ such that $(F_*,\cX_*)$
%is $(H,\cY_*,\cV_*)$ are compatible the
we obtain from the  induction hypothesis  the following for the  second sum in \eqref{eq:T*}: 

\begin{align}\label{eq:T*small}
\sum_{\substack{T_*\subseteq H\\ |\ext(T_*)|<\eps' z}}(1+\eps')\eps' pz&\leq 2 c^{e(F_*)-t_*}p^{e(F_*)} |U|^{v(F)-k-r}\cdot (1+\eps')\eps'  pz 
\leq\frac\gamma4 c^{e(F)-t}p^{e(F)} |U|^{v(F)-r}.
%\vol(V_{1},\dots, V_{f}).
%\leq  \frac{2\eps'}{\gamma^k} p^{e(F)}\vol(V_1,\dots,V_f)<\frac\gamma4 p^{e(F)}\vol(V_1,\dots,V_f).
\end{align}
%We note  that  $e(F_*)-t_*\leq e(F)-t$ is possible
%as edges with root vertices may contain one of the~$u_i$. Thus we use in \eqref{eq:T*small}, that we choose $\eps$ to be small compared to $c$.

To derive a bound for the first sum in \eqref{eq:T*}, we will split the sum further.  Define $J:=\log 1/\eps' +w_{q}\log 1/p$ and for all $0\leq j\leq J$, let $b_j$ be the number of copies $T_*$ of $F_*$ in $H$ such that $2^j \eps' z\leq|\ext(T_*)|\leq 2^{j+1}\eps' z$. Note that this covers  all possible copies as $2^{J+1}\eps' z\geq |U|^k$. 
Then the  number of rooted copies of $F'$ in $H$ is at least
$\sum_{j=0}^Jb_j2^j\eps' z$ and, by induction hypothesis, at most $2p^{e(F')}|U|^{v(F')-r}= 2p^{e(F)-1}|U|^{v(F)-r}$. 
%Hence $\sum_{j=0}^Jb_j2^j \eps' z\leq 2 p^{e(F')}|U|^{f'}$.
%\[
%\sum_{j=0}^Jb_j2^j \eps' z\leq 2 p^{e(F')}|U|^{f'}.
%\]
%$b_j$ is at most $\frac1{2^{j-1}\eps'}p^{e(F_*)}  |U|^{f'-k}$ for any  $j\in[J]$.
%$\frac2{(ac^{e(F_*)})}(cp)^{e(F_*)}  \vol(V_{k+1}, \dots, V_{f})$ for any  $a>0$. 
%Otherwise, by the inductive hypothesis, the number of copies of $F_*$ in~$H$ is at least $(1-\r/2)(cp)^{e(F_*)}  \vol(V_{k+1}, \dots, V_{f})$ and 
%Indeed, otherwise  the number of copies of $F'$ in~$H$ would be at least
%\begin{align*}
%\frac1{2^{j-1}\eps'}p^{e(F_*)}  |U|^{f'-k}\cdot 2^j \eps' z \geq 
%2p^{e(F')}  |U|^{f'},
%\end{align*}
%which  contradicts the inductive hypothesis on the counting of $F'$.
%Hence, for each integer $j\geq 0$ we have %the contribution of those $T_*\subset H$ with $2^j\eps' z<|\ext(T_*)|\leq 2^{j+1}\eps' z$ to the first sum in \eqref{eq:T*} is
%\[ \sum_{\substack{T_*\subseteq H\\2^j\eps' z\leq|\ext(T_*)|\leq 2^{j+1}\eps' z}} \eps' p|\ext(T_*)|< \frac2{2^j\eps' } p^{e(F_*)} \vol(V_{k+1}, \dots, V_f)\cdot \eps' p 2^{j+1}\eps' z< 8\eps' p^{e(F)}\vol(V_1,\dots, V_f).\]
%Let $\ell=\log 1/\eps' +w_{\sigma(t)}\log 1/p$ so that $2^\ell\eps' z\geq \vol(V_1,\dots, V_k)$ and $(\ell+1)\eps' \leq c^{e(F)}\gamma/32$. 
Consequently, the first sum  in~\eqref{eq:T*} is
\[\begin{aligned}
 \sum_{\substack{T_*\subseteq H\\ \eps' z \leq |\ext(T_*)|\leq |U|^k}} \eps' p|\ext(T_*)|&\leq \sum_{j=0}^{J} \sum_{\substack{T_*\subseteq H\\2^j\eps' z\leq|\ext(T_*)|\leq 2^{j+1}\eps' z}} \eps' p|\ext(T_*)|\\& \leq \sum_{j=0}^Jb_j2^{j+1}\eps'^2pz \leq 4\eps' p^{e(F)}|U|^{v(F)-r} \leq \frac{\gamma}4 (cp)^{e(F)}|U|^{v(F)-r}.
\end{aligned}\]
Together with \eqref{eq:T*} and \eqref{eq:T*small} we conclude that $|\sum_{T\subseteq H} (\vecone_H(K_T) - p)|<\frac{\gamma}2 c^{e(F)-t}p^{e(F)}|U|^{v(F)-r}$
which finishes the proof of the lemma.
\end{proof}

\begin{remark}\label{rem:tightcount}
Note that if we take $r=0$ in Lemma~\ref{lem:rootcount}, we can drop the maximum co-degree assumption as it is not required in the proof and so this  gives Theorem~\ref{prop:small subgraph appearence}. In fact, the proof of Lemma~\ref{lem:rootcount} actually establishes the stronger bounds~\eqref{eq:rootcount}, which under the same pseudo-randomness condition on $H$, yields the counting property for linear $k$-graphs. 
In general, the  condition is tight up to a multiplicative constant as seen,
e.g., by Alon's construction~\cite{A94} of triangle-free {$n$-vertex $d$-regular graphs with $d=\Omega(n^{2/3})$, which are $(d/n,\lambda)$-jumbled with  $\lambda=O(n^{1/3})$.}
On the other hand the bound can be improved for other graphs, e.g., when $F$ is a larger odd cycle (see \cite{KrivelevichSudakov}).
\end{remark}

%The following lemma will be useful to find absorber family, connecting, covering sets with paths.
\subsection{Embedding compatible families} \label{subsec:compatiblefams}
%Lemma \ref{lem:rootcount} shows that for a given rooted hypergraph $F$, if $H$ is sufficiently pseudo-random with respect to $F$, then we can find the same number of copies of $F$ in $H$ as we would expect in a random hypergraph of the same density as $H$. 
In this section we  use the counting lemma, Lemma~\ref{lem:rootcount}, to build a \emph{linear sized} structure in the host $k$-graph.
This structure is key to the absorption step but will  also be useful in other parts of the proof. 
Let $(A,\cX)$ be a fixed rooted hypergraph with~$r$ root vertices. 
The structure we look to find 
will consist of many  rooted copies of $A$ in~$H$, which respect a certain intersection restriction on the root vertices but  {are} disjoint otherwise.
%  given by an auxiliary hypergraph. %, which controls the intersections of the copies of $A$ we find. 
Formally, let $T=(V_T,E)$  be a labelled $r$-graph with the vertex set $V_T\subset V(H)$ (which captures the intersection structure)\footnote{Note that each edge $e\in E=E(T)$ is a labelled $r$-set in $V_T\subset V(H)$.}.
Then  $\{A_{e}\}_{e\in E(T)}$   is called {a \emph{$T$-compatible family of copies of  $(A,\cX)$}} {(or simply $T$-compatible)} if:
\begin{enumerate} \item each $A_e$, $e\in E(T)$, is a 
rooted copy of $(A,\cX)$ in $(H,e,V(H))$, i.e., a rooted copy of $A$ in $H$ which maps $\cX$ to $\cY=e$;
\item each $A_e$, $e\in E(T)$, intersects $V_T$ exactly in $e$; and
\item for any two edges $e,e'\in E(T)$,  the copies~$A_{e}$ and $A_{e'}$ intersect exactly in $e\cap e'$. 
\end{enumerate}
In particular, note that  the copies $A_{e}$ and $A_{e'}$, $e\neq e'$, are  disjoint {outside of $V(T)$. } Our next lemma shows that if $T$ has bounded degree and our host hypergraph $H$ is suitably pseudo-random with respect to some $k$-graph $A$, then we can find some $T$-compatible family of copies of $A$. 
%As an example, one could take $(A,\cX)$ to be a single $k$-edge with a single root vertex. Then if we take $T$ to be a (1-uniform) set of $t$ vertices, $V_T\subset V(H)$, a $T$-compatible family $\{A_e\}_{e\in T}$ is simply a set of $t$ disjoint edges in $H$, each of which contains exactly one vertex of $V_T$.

%The concept is much more general than this and the following lemma will be useful in several contexts.
\begin{lemma}
\label{lem:absorber}
Given integers $k\geq 2, f, r\geq 0$, $\Delta\geq 1$  and $c>0$ there are $\eps>0$ and $n_0$ such that for $n>n_0$ the following holds.
Suppose that 
\begin{itemize}\item $(A,\cX)$ is a rooted linear $k$-graph with  $r+f$ vertices, $r$ of which are roots, 
\item $H$ is a $(p,\eps p^{\ell},\eps)$-pseudo-random $n$-vertex $k$-graph, with $\ell\geq \degen{(A,\cX)},$ and $\Delta_2(H)<\eps pn^{k-1}$, %with  $\Delta_2(H)\leq C {p^{1-\degen A} n^{k-2}}$and 
 $Y\subset~V(H)$ with $|Y|\leq\frac{n}{200\Delta^2 (r+f)^2}$ and $\deg_H(v;V\setminus Y)>cpn^{k-1}$ for all $v\in Y$,
 \item  $T$ is an ordered $r$-graph on the vertex set $V_T=Y$ with maximum vertex degree $\Delta_1(T)\leq \Delta$.
\end{itemize}
Then there exists a $T$-compatible family of rooted copies of $(A,\cX)$ in $H$. 
\end{lemma}
We note here again that by Lemma~\ref{lem:lowerboundonp} the condition  $\Delta_2(H)<\eps pn^{k-1}$ can be dropped if $\ell\geq k-1$.

\begin{proof}

Let integers $k,f,r, \Delta$  and $c$ be given. We choose $\gamma= \frac{1}{100 \Delta f}$ and 
as $c>0$ only appears in the lower bound for the degree condition we may assume that $c< { \gamma/8}$.
We apply the counting lemma, Lemma~\ref{lem:rootcount}, with the parameters $k,f,r, c_{\ref{lem:rootcount}}=\frac c{2^k}$
to obtain $\eps_{\ref{lem:rootcount}}$.
We choose $\eps =\eps_{\ref{lem:rootcount}}\frac{(c\gamma)^{k}}{ 4^{k}f\Delta}$  and let~$n_0$ be sufficiently large.
Let $(A,\cX)$, $H$ and $T$ with $V_T=Y$ be  as in the lemma. 

The idea is to construct the required family of rooted copies of $(A,\cX)$ by  repeatedly using the greedy type Algorithm~\ref{alg:absorber}, which simply
extends the family of rooted hypergraphs when it can and records the failure otherwise.
 \begin{algorithm}[t]
    \caption{Greedy builder}\label{alg:absorber}
    \SetAlgoLined
    \SetKwInOut{Input}{Input}
    \Input{ $(A,\cX)$, $H$, $\hat T$, $X\subset V(H)$;}
%    \SetKwInOut{Output}{Output}
  %  \Output{$\cA=\{A_e\}_{e\in E(T)}$ family  of rooted copies $A_e$ of $(A,\cX)$ in $(H,e,\cV)$\;}
   Let $(e_1,\dots,e_t)$ be an ordering of $E(\hat T)$\;
    $X^1:=X$, $I^0:=\emptyset$,   $\cA^0:=\emptyset$, and  $s:=1$\;

   \While{$s\leq t$}  {
   	\uIf {there is   a rooted copy $A_{s}$ of $(A,\cX)$ in $(H,e_s,X^s)$}
   		{$\cA^{s}:=\cA^{s-1}\cup\{A_{s}\}$\;
		  $I^{s}:=I^{s-1}$\;

    		$X^{s+1}:= X^{s}\setminus V(A_s)$\; 
    	}
		\uElse{$I^{s}:=I^{s-1}\cup\{e_s\}$\;}
			$s:=s+1$\;

    }
  \end{algorithm}

This greedy approach almost succeeds in finding all the copies of $(A,\cX)$ but we can not guarantee that it will log no failures. To deal with the small number of failures, we have to run the algorithm several times. First, we analyse the following simple case  where we can successfully embed a small number of copies. \begin{claim}\label{claim:Tsmall}
Suppose that  $\hat T\subset T$ is a subhypergraph of $T$ with $e(\hat T)\leq \frac{c^k}{4f}pn$ and $X\subset V(H)\setminus V_T$ is a set of size $|X|\geq c n$  
so that $\deg(v;X)>cp|X|^{k-1}$ holds for each $v$ contained in an edge of $\hat T$. 
Then there  is a $\hat T$-compatible family of rooted copies of $(A,\cX)$ whose vertices are entirely contained in $V_T\cup X$.
\end{claim}
\begin{proof}[Proof of Claim~\ref{claim:Tsmall}]
We run the Algorithm~\ref{alg:absorber} with input $(A,\cX)$, $H$, $\hat T$ and $X$ and claim that the family $\cA^t$, $t=e(\hat T)$, produced by the algorithm, has the required properties.
Let $(e_1,\dots,e_t)$ be an ordering of $E(\hat T)$. Note that after step $s\in[t]$ the algorithm 
has removed from $X$ in total at most $f \cdot s$ vertices. Thus, at each time $s$ the set $X^s$ in the algorithm has size $|X^s|\geq \frac12|X|> c_{\ref{lem:rootcount}} n$ and  
$\deg(v;X^s)\geq \deg(v;X)-f s \cdot\Delta_2(H)>c_{\ref{lem:rootcount}} p |X^s|^{k-1}$ holds for each $v\in e_{s}$, using that $\Delta_2(H)\leq n^{k-2}$.
%In particular, by Fact~\ref{fact:subset}, 
 %$H[X^s]$ is $(p,2^k\eps \alpha^{-k} p^{\ell },\eps)$ 
 %\frac \eps{\log (1/p)})$
 %-pseudo-random and %with the condition on $\Delta_2(H)$ 
 % we have that $(H,e_s,X^s)$ is $(c/2,p,\eps)$-compatible with $(A,\cX)$. 
 %Let $\sigma$ be an edge exposure which certifies the degeneracy of $A$.
%Then 
Thus, Lemma~\ref{lem:rootcount} applied with the choices of constants 
%and $\sigma$ 
yields a rooted copy $A_s$ of $(A,\cX)$ in $(H,e_s,X^s)$. As this holds for all $s\in[t]$ the family $\cA^t$  is $\hat T$-compatible and is contained in $V_T\cup X$.
\end{proof}
Returning to the proof of Lemma \ref{lem:absorber}, choose disjoint subsets $U, W\subset V(H)\setminus Y$ of size $|U|=|W|=2\gamma n$ but arbitrarily otherwise. 
Let \[B=\{v\in V_T\colon \deg(v;W)<2cp|W|^{k-1}\}\] 
and let $T_1\subset T$ denote the subgraph of $T$  on the same vertex set $V_T$ which consists of all edges {intersecting $B$}.
    By {the} pseudo-randomness of $H$ we conclude that $|B|\leq \eps \gamma^{-(k-1)} p^{\ell}n$. As  $e(T_1)\leq \Delta |B|$ 
we  can apply Claim~\ref{claim:Tsmall} with $\hat T=T_1$ and $X=V(H)\setminus V_T$
to find a $T_1$-compatible family $\cA_1$. 
%Note that here, we consider $V_{T_1}=V_T$ so that we guarantee that $A_e\cap V_T=e$ for all $e\in T_1$. This also holds for the subgraphs $T_2$ and $T_3$ of $T$ below.
Let\footnote{Here, $V(\cA)$ denotes the set of vertices which feature in rooted copies of $A$  in the family $\cA$.}
 $U'=U\setminus V(\cA_1)$ and $W'=W\setminus V(\cA_1)$ which are disjoint sets of size $|U'|,|W'|\geq \gamma n$ and
note that  each vertex $v\in V_T\setminus B$ satisfies
 
  \begin{align}\label{eq:degW'}\deg(v;W')\geq\deg(v;W)-f\cdot e(T_1)\Delta_2(H)>cp |W'|^{k-1}.\end{align}
\begin{claim}\label{claim:Tlarge}
Let $T'\subset T$ be the subgraph obtained by removing the edge set $E(T_1)$ from $T$.
Then there is a subgraph $T_2\subset T'$  with $e(T_2)\geq e(T')-(2/\gamma)^{k-1}\Delta\eps p^{\ell}n$ and a $T_2$-compatible 
family $\cA_2$ of rooted copies of $(A,\cX)$ whose vertices are entirely contained in $V_T\cup U'$.
\end{claim}
Before proving the claim we note that it readily implies the lemma. Indeed, define $T_3= T'\setminus T_2$ which then satisfies $e(T_3)\leq (2/\gamma)^{k-1}\Delta\eps p^{\ell}n$. 
Further,~\eqref{eq:degW'} holds for all vertices in $V_T\setminus B$, in particular for all those contained in edges of $T_3$.
Thus we can apply Claim~\ref{claim:Tsmall} with $\hat T=T_3$ and $X=W'$   and obtain a $T_3$-compatible family 
$\cA_3$ of rooted copies of $(A,\cX)$ whose vertices are entirely contained in $V_T\cup W'$. Since $T=T_1\cup T_2\cup T_3$, the family 
 $\cA_1\cup \cA_2\cup \cA_3$ is   $T$-compatible 
and the lemma follows.
\end{proof}

\begin{proof}[Proof of Claim~\ref{claim:Tlarge}]
We run the Algorithm~\ref{alg:absorber} with input $(A,\cX)$, $H$, $\hat T=T'$ and $X= U'$ and it is sufficient to show that {$|I^t|\leq (2/\gamma)^{k-1}\Delta\eps p^{\ell}n$.} Let $(e_1,\dots, e_t)$ be an ordering of $E(T')$.
As $t< \Delta |V_T|\leq \frac{\gamma n}{2f}\leq \frac{|X|}{2f}$ and at each time $s< t$ the algorithm removes at most  $f$ vertices from $X^s$ to obtain $X^{s+1}$, we have for each time $s\in[t]$   that 
$|X^s|\geq |X^1|-(s-1)f>|X^1|/2 \ge \gamma n/2$.
% Fact~\ref{fact:subset} then implies that  $X^s$ is $(p, {2^k\alpha^{-k}}\eps p^{\ell},\eps)$-pseudo-random for each $s\in[t]$.
 %Let $\sigma$ be an edge exposure which certifies the degeneracy of $A$.
Then Lemma~\ref{lem:rootcount} applied with the choices of constants 
%and $\sigma$ 
implies that for each $s\in[t]$ there is a rooted copy of $(A,\cX)$ in $(H,e_s,X^s)$
unless  the  degree condition fails for a  vertex in~$e_s$. %of Lemma \ref{lem:rootcount}. %to be $(c',p,\eps')$-root-compatible with $(A,\cX)$. 
Thus, $I^t$ comprises exactly of those  $e_s$ such that for some vertex in $e_s$, say $y^s$, we have $\deg(y^s;X^s)<c_{\ref{lem:rootcount}}p|X^s|^{k-1}\leq c_{\ref{lem:rootcount}}p|X|^{k-1}$,
%which indicate  failure of root-compatibility for $(H,e_s,X^s)$. 
 which  implies that  $\deg(y^s;X^t)<c_{\ref{lem:rootcount}}p|X|^{k-1}\leq cp|X^t|^{k-1}$.
%In particular, for each $s\in I$ there is a $y^s\in e_s\in E(T)$ which certifies the failure of compatibility. More explicitly, this means that   there is an edge 
Let \[Y^t=\{y\in V_T\colon\deg(y;X^t)<cp|X^t|^{k-1}\}.\]
Clearly,  $e(Y^t,X^t,\ldots,X^t)\leq cp|Y^t||X^t|^{k-1}$ and the  pseudo-randomness  condition together with $|X^t|\geq \gamma n/2$ implies that $|Y^t|\leq (2/\gamma)^{k-1}\eps p^{\ell}n$.
On the other hand,  for any $e_s\in I^t$, the vertex $y^s\in e_s$, as detailed above, is contained in
$Y^t$ and every $y\in Y^t$ is $y^s$ for at most $\Delta$ elements $e_s\in I$. This shows that $|I^t|\leq \Delta|Y^t|$ and the claim follows.
\end{proof}

%\subsection{An application}
%We conclude this section with a simple application of Lemma \ref{lem:absorber} which will be useful in proving the main theorems. 

\section{Absorbers and the template}\label{sec:absorber}

%Given a linear $k$-graph $F$ on the vertex set $V(F)=[f]$ 
%we define the \emph{grid-absorber of $F$} to be the  linear $k$-graph~$A_F$ on the vertex set $[f]\times[f]$ as follows.
%For each $i\in[f]$ the absorber contains a copy $F_i$  of $F$ on the vertex set $V(F_i)=\{i\}\times[f]$.
%Further, for each $i\in[f-1]$ the absorber contains a copy $F^i$ of $F$ on the vertex set 
%$V(F^i)=[f]\times\{i\}$. 
As mentioned in  Section~\ref{sec:outline} our proof works by absorption. In particular, it utilises the ``absorber-template'' method introduced by Montgomery \cite{M14a,montgomery2019spanning} which 
has since been used by various authors in different settings, see, e.g., \cite{FKL16,FN17,HKMP18,han2019finding,han2019tilings,Kwan16,nenadov2018ramsey}. 
In our case we combine many copies of small special subgraphs called  \emph{absorbers} to a large family using a large \emph{template}, which captures how these copies intersect.
%The template with its key properties can be found in Section~\ref{}
The absorbers will depend  on the spanning structure we are interested  in but in
 both cases we should be able to find them in the pseudo-random host, thus
they need to be linear hypergraphs and  edges containing distinct rooted vertices should be disjoint. 
Moreover, in light of Lemma~\ref{lem:rootcount} we want that absorbers have small edge degeneracy as to weaken the pseudo-random condition necessary for the argument.  
The absorbers for factors {given here are  straightforward to describe and rely on permuting  copies of $F$ in a grid-like structure so as to reduce the degeneracy.} 
The path absorbers, however, are more involved and  differ from those used in absorbing arguments before. 
In particular, the first absorbers \cite{HanSchacht,BHS}  for finding loose Hamilton cycles in hypergraphs  do not satisfy our definition of path absorbers, 
while the ones used, e.g., in~\cite{LMM},  have edge degeneracy~$k$ and so are not as effective as those given here, which have degeneracy $k-1$.%, for our needs.

%The idea is to find an absorbing structure in the host hypergraph which can contribute to the desired spanning structure in many different ways. 
%This gives flexibility when trying to find a given spanning structure, and will allow us to complete an `almost spanning' structure to a full structure 
%which covers the whole vertex set of the host hypergraph. 
%Our absorbing structure will be defined by means of an auxiliary graph called a template. 
%First though, we define small subgraphs which we call \emph{absorbers}. These will play the role of the ``building blocks'' for the absorbing structure. 
\subsection{Absorbers for factors} \label{subsec:factorabsorbers}

For a linear $k$-graph $F$ an \emph{absorber} for $F$  is a rooted  linear $k$-graph $(A_F,\cX)$ with non-empty root vector~$\cX$ such that there is an
$F$-factor of $A_F$ and an $F$-factor of the subgraph of $A_F$ obtained by removing all  roots $\cX$. We will often refer to the $F$-factor on the full vertex set of $A_F$ as the \emph{complete} $F$-factor, while the factor on $V(A_F)\setminus \cX$ is referred to as the \emph{internal} $F$-factor.  Note that the number of root vertices is a multiple of $|V(F)|$.

\begin{lemma}\label{lem:AF}
Given $k$ and a linear $k$-graph $F$ on  $[f]$ vertices.
Then there is an absorber $(A_F,\cX)$ of $F$ with~$f^2$ vertices,   $f$  roots and edge degeneracy   
\[\degen(A_F, \cX)\leq \degen(F)+\max_{e\in E(F)}\sum_{u\in e}\deg_F(u)= \degen(F)+\Delta'(F)+k,\]
where $\degen(F)$ is considered here to be the degeneracy of $F$ with an empty root set of vertices. 
In particular, if $F$ consists of a single edge, then there is an absorber of edge degeneracy at most $k$.
\end{lemma}
\begin{proof}
Given a linear $k$-graph $F$ on the vertex set $V(F)=\Zf=\mathbb Z/f\mathbb Z$ we define
the $k$-graph~$A_F$ on the vertex set $\Zf\times\Zf$ as follows.
For each $i\in\Zf$, $A_F$ contains a copy $F_i$  of $F$ on the vertex set $V(F_i)=\{i\}\times\Zf$.
Further, for each $j\in[f-1]$, $A_F$ contains a copy $F^j$ of $F$ on the vertex set 
$V(F^j)=\{(i,i+j)\colon i\in\Zf\}$ where addition is in $\Zf$. These copies we place so that the projection to the second coordinate $\phi((\cdot,\ell))=\ell$
defines an isomorphism  between~$F_i$ and~$F$ and between $F^j$ and~$F$, respectively. 
Note that $A_F$ is a linear hypergraph since
$F_i$ and $F_j$ are disjoint for $ij\in\Zf\times \Zf$, $i\neq j$, and $F^i$ and $F^j$ are disjoint for $ij\in[f-1]\times [f-1], i\neq j$, 
while $F_i$ and $F^j$,  $ij\in\Zf\times[f-1]$ intersect only in the vertex $(i,i+j)$. We further define roots $\cX=((1,1),\dots, (f,f))$ and obtain the rooted $k$-graph $(A_F,\cX)$ with the property that $\{F_i\}_{i\in\Zf}$ gives the complete $F$-factor of   $A_F$ while $\{F^j\}_{j\in[f-1]}$
gives the internal $F$-factor.
 Hence, $A_F$ is an absorber of $F$ and it remains to show the bound on the degeneracy of $A_F$. 
 
Let $\sigma$ denote an edge exposure of $F$ (without any roots) which yields the degeneracy of $F$.
% which we will use  to expose
%the edges of $F^j$. 
We construct an edge exposure $\tau$ for $A_F$ by first exposing edges containing the roots in an arbitrary order. 
Note that this step does not expose any edge of  $F^j$, $j\in[f-1]$, since none of them contains root vertices. % while each $F_i$, $i\in\Zf$ contains the root $(i,i)$.
In the second step we expose the remaining edges of 
all $F_i$, $i\in\Zf$
in an arbitrary order and finally, in the third step, we expose the edges of each~$F^j$, $j=1,\dots, f-1$ according to the order given by $\sigma$. 
As the $F_i$'s are vertex disjoint each edge $e$ from the  second step  has weight at most $\sum_{u\in e} (\deg_F(u)-1)$.
 % and  can therefore be considered separately. 
 
To bound the weights of the edges in the third step consider a $j\in[f-1]$ and 
let  $e_t=\{(i_1,i_1+j),\dots,(i_k,i_k+j)\}$ be the $t$-th edge of $F^j$ in the ordering~$\tau$. 
Recall that $F^j$ is disjoint from other $F^{j'}$, that we expose the edges of~$F^j$ according to $\sigma$  and  that $V(F_i)\cap V(F^j)=\{(i,i+j)\}$  for $i\in\Zf$. 
Therefore, the weight of $e_t$ with respect to $\tau$ is exactly \[w_{\sigma(t)}+\sum_{\ell\in[k]}\deg_{F_{i_\ell}}((i_\ell,i_\ell+j)).\]
%the weight of $e_t$ within $F^j$ alone
%(at the time it gets exposed)
%is $w_{\sigma(t)}\leq \degen F$ as $F^j$ is isomorphic to~$F$. It remains to estimate the contributions of the $F_i$'s to the weight of $e_t$ and we 
Recall also that  $F_i$ and $F^j$ are identical copies of $F$, i.e., the projection to the second coordinate $\phi(\cdot,\ell)=\ell$ is an isomorphism of $F_i$ and $F^j$ to  $F$. 
By this projection  $\{i_1+j,\dots, i_k+j\}$ is an edge in $F$ and the degree of $(i_\ell,i_\ell+j)$ in $F_{i_\ell}$ is $\deg_F(i_\ell+j)$ for each $\ell\in[k]$.
Hence, $\sum_{\ell\in[k]}\deg_{F_{i_\ell}}((i_\ell,i_\ell+j))\leq \max_{e\in E(F)}\sum_{u\in e}\deg_F(u).$ 
Together with $w_{\sigma(t)}\leq \degen (F)$ this yields the desired bound.
\end{proof}

\subsection{Path absorbers} \label{subsec:pathabsorbers}
A $k$-uniform \emph{path absorber} is a rooted linear $k$-graph $(P,\cX,y_1,y_2)$ with a non-empty set of root vertices $\cX$ and two distinguished vertices $y_1,y_2\in V(P)\setminus \cX$, 
called \emph{end vertices},
 such that there is a loose path from $y_1$ to $y_2$ which uses all the vertices of $V(P)$ and a loose path from $y_1$ to $y_2$ in $P$ 
 which covers the vertices $V(P)\setminus \cX$. The first path we call the \emph{complete} loose path and latter is called the \emph{internal} loose path.

\begin{lemma} \label{lem:AHC3}
For each $k\geq 3$, there exists a $k$-uniform path absorber $(P,\cX,y_1,y_2)$ on $9k^2-23k+15$ vertices with roots $\cX=\{x_1,\ldots,x_{k-1}\}$ and degeneracy $\degen (P,\cX)=k-1$.
\end{lemma}
\begin{proof}
Our path absorber is defined by smaller subgraphs which we call \emph{absorbing gadgets} (see Figure \ref{fig:absorbing_gadget}). 
An absorbing gadget $P_i$ is a hypergraph on $5k-6$ vertices, with disjoint vertex subsets $A_{i}$, $A'_i$, $B_i$ and $C_i$ such that $|A_i|=|A_i'|=k-2$, $|B_i|=k-3$, 
$|C_i|=2k+1$ and $V(P_i)=A_i\cup A'_i\cup B_i \cup C_i$. We  label $C_i=\{c_{i1},\ldots,c_{ik},c'_{i1},\ldots,c'_{ik},c_{i*}\}$ and let $E(P_i)=\{e_i,e_i',f_i,f_i',g_i\}$ where
% \vspace*{-\baselineskip}
\[\begin{aligned}
e_i:=\{c_{i2},c_{i*}\}\cup &A_i,\qquad  f_i:=\{c_{i1},\ldots,c_{ik}\},\\  
e'_i:=\{c_{i2}',c_{i*}\}\cup &A_i',\qquad f_i':=\{c_{i1}',\ldots,c'_{ik}\} \qquad { and }\\ g_i:=&\{c_{ik},c_{i*},c_{i1}'\}\cup B_i.\end{aligned}\]
The key property of the absorbing gadget is that there are two loose paths in $P_i$, both with end vertices $c_{i1}$ and $c'_{ik}$. 
Indeed, there is a loose path from $c_{i1}$ to $c'_{ik}$ which covers $V(P)\setminus (A_i\cup A_i')$, namely taking the edge sequence $(f_i,g_i,f_i')$. 
We call this the \emph{inner loose path} of $P_i$. We also have the \emph{outer loose path} of~$P_i$ defined by the edge sequence $(f_i,e_i,e_i',f'_i)$ which covers $V(P)\setminus B_i$ and also has endpoints $c_{i1}$ and $c_{ik}'$. 

\begin{figure}
    \centering
  \includegraphics[scale=.8]{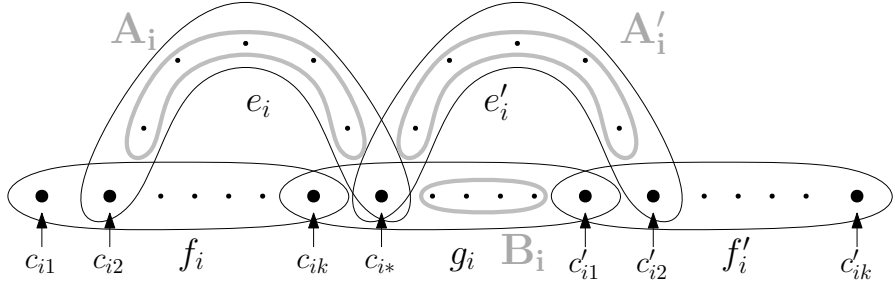}
    \caption{An absorbing gadget.}
    \label{fig:absorbing_gadget}
\end{figure}

The path absorber $P$ is then defined by taking copies of these absorbing gadgets, joining them together using singular edges 
and identifying  vertices  in $A_i\cup A_i'\cup B_i$ across the gadgets. In more detail, we take the vertex set of $P$ to 
be the disjoint union \[V(P):=\cX \cup U\cup V \cup W\cup \bigcup_{i=1}^{2k-3}C_i\cup \bigcup_{i=1}^{2k-4}D_i.\]
On $V(P)$ we will define absorbing gadgets $P_1,\dots, P_{2k-3}$, where $V(P_i)=A_i\cup A'_i\cup B_i \cup C_i$ and $E(P_i)=\{e_i,e'_i,f_i,f'_i,g_i\}$ as above.
The set $C_i$ is also labelled as in the previous paragraph and is used by $P_i$ only, while vertices in $A_i, A_i'$ and $B_i$ will be shared with other gadgets in the way 
explained in the next paragraph.
 The sets~$D_i$, $|D_i|=k-2$, {are used} to build the edges $h_i=\{c_{ik}',c_{(i+1)1}\}\cup D_i$, $i\in[2k-4]$ which  connect~$P_i$ and $P_{i+1}$.
 All together we obtain edge set of $P$ given by 
\[E(P):= E(P_1)\cup\dots\cup E(P_{2k-3})\cup\{h_1,\dots, h_{2k-4}\}\]
and we take $y_1:=c_{11}$ and $y_2:=c'_{(2k-3)k}$ to be the two endpoints of $P$.

To complete the definition of $P$, it is left to assign the vertices in $A_i,A_i'$ and $B_i$ as subsets of $\cX \cup U\cup V \cup W$. For this purpose
we consider the following labelling % the   sets $\cX$ of size $k-1$, $U$ of size $(k-1)(k-3)$, $W$ of size $(k-2)(k-3)$ and $V$ of size $(k-1)(k-2)$ as follows
\[\begin{aligned}
    \cX &:= \{x_1,\ldots,x_{k-1}\}, \\
    U &:=\{u_{ij}\colon1\leq i, j \leq k-1~\text{ and }~i \bmod (k-1)\notin\{ j,j-1\} \},\\ 
    W &:=\{w_{ij}\colon k\leq i, j \leq 2k-3~\text{ and }~i\neq j \},\\ 
    V&:=\{v_{ij}\colon 1\leq i \leq k-1~\text{ and }~k \leq j \leq 2k-3\}.
\end{aligned}\]
%$\cX$ has size $k-1$, 
%$U$ has size $k^2-4k+3=(k-3)(k-1)$,~$W$ has size $k^2-5k+6=(k-2)(k-3)$ and $V$ has size $k^2-3k+2=(k-1)(k-2)
%$.  We will define, for each $i\in[2k-3]$, an assignment of vertex sets $A_i,A_i',B_i \subset \cX \cup U\cup V \cup W$. Then the edge set of $P$ is given by the edges of each of the absorbing gadgets as well as the connecting edges $h_i$. That is, 
%\[E(P):= E(P_1)\cup\dots\cup E(P_{2k-3})\cup\{h_1,\dots, h_{2k-4}\}\quad\text{where}\quad E(P_i)=\{e_i,e'_i,f_i,f'_i,g_i\}.\]
We assign the vertices as follows. For $1\leq i \leq k-1$ we define 
\[\begin{aligned}
    A_i&:= \{x_i\}\cup\{u_{i\ell}\colon1\leq \ell \leq k-1~\text{ and }~\ell \bmod (k-1) \notin \{i,i+1\} \}, \\
    A_i'&:=\{v_{i\ell}\colon k \leq \ell \leq 2k-3\},\\
    B_i&:=\{u_{\ell i}\colon1\leq \ell \leq k-1~\text{ and }~\ell \bmod (k-1) \notin \{i-1,i\}  \}.
\end{aligned}\]
On the other hand{,} for $k\leq i \leq 2k-3$, we assign the vertices in the following way:
\[\begin{aligned}
    A_i&:= \{v_{(i-k+1)i}\}\cup\{ w_{i\ell}: k\leq \ell \leq 2k-3~\text{ and }~\ell \neq i\}, \\
    A_i'&:=\{v_{\ell i}\colon 1 \leq \ell \leq k-1~\text{ and }~\ell \neq i-k+1)\}\qquad \mbox{ and }\\
    B_i&:=\{w_{\ell i}\colon k\leq \ell \leq 2k-3~\text{ and }~\ell \ne i \}.  
    \end{aligned}\]
This finishes the definition of $P$.    
 Examples of these vertex set assignments are shown in Figure~\ref{fig:assignment}. When $k=3$, the sets $B_i$ are empty which simplifies the situation considerably. The resulting hypergraph in this 
  case is shown in Figure \ref{fig:three_unif}.

\begin{figure}
    \centering
  \includegraphics[scale=0.7]{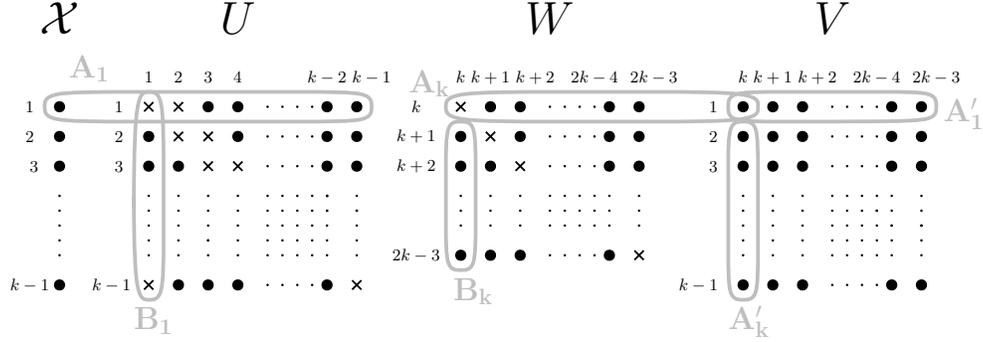}
    \caption{Assigning the vertices.}
    \label{fig:assignment}
\end{figure}

Let us now establish the claimed properties. Firstly, it is easily verified that $V(P)$ has the required size and  we now show that $P$ is linear. Note that for a fixed $i\in[2k-3]$, the vertices that appear in the absorbing gadget $P_i$ are all distinct. 
One can easily check that~$P_i$ is linear and the edges $h_i$ intersect two distinct gadgets in one vertex each. 
Therefore it suffices to establish that if $e\in E(P_i)$ and  $f\in E(P_{j})$, for $i\neq j$ then $|e\cap f|\leq 1$. 
If $i,j\in[k-1]$, $i\notin \{j-1,j+1\} \bmod(k-1)$, then $V(P_i)\cap V(P_{j})=\{u_{ij},u_{ji}\}$ and these vertices do not lie in the same edge, neither in $P_i$ nor in~$P_{j}$. 
If  $i\in \{j-1,j+1\} \bmod(k-1)$, then the situation is even simpler as $|V(P_i)\cap V(P_{j})|\leq 1$. 
In a similar fashion, we have that if $k\leq i,j\leq 2k-3$ then $V(P_i)\cap V(P_{j})=\{w_{ij},w_{ji}\}$ and neither $P_i$ nor~$P_{j}$ have an edge which contains both of these vertices. 
Finally if $1\leq i \leq k-1$ and $k\leq j\leq 2k-3$, then $V(P_i)\cap V(P_{j})=\{v_{ij}\}$.

We now verify  that $(P,\cX,y_1,y_2)$ defines a path absorber. 
Consider traversing the outer loose path of $P_i$ for $1\leq i\leq k-1$ and the inner loose path for $k\leq i \leq 2k-3$ as well as the edges~$h_i$ which link the absorbing gadgets. 
This gives a loose path from $y_1$ to $y_2$, which covers all the vertices exactly once. 
Indeed, when we traverse the outer loose paths in the first $k-1$ absorbing gadgets we cover all of $U$, $V$  and 
$\cX$ exactly once whilst avoiding the $B_i$ and thus avoiding taking any $u\in U$ more than once. 
Further, taking the inner loose paths for $k\leq i \leq 2k-3$ guarantees that we cover~$W$ and do not repeat any vertices of $V$ in the process.  
Similarly, if we now consider taking the inner loose path for all the absorbing gadgets $P_i$ with $1\leq i\leq k-1$ and the outer loose path for all 
absorbing gadgets $P_i$ with $k\leq i \leq 2k-3$, one can see that we obtain a loose path from $y_1$ to $y_2$ on $P$ which covers exactly the vertices $V(P)\setminus \cX$.

Finally let us turn to the degeneracy of $(P,\cX)$. We consider the following edge exposure. We reveal all the $e_i$ for $i=1,2,\ldots,2k-3$ first. 
This guarantees that all the edges with root vertices are revealed first. 
We then reveal $e'_i$ for  $i=1,2,\ldots,2k-3$. 
Each $e'_i$ has weight at most {$k-1$}. Indeed, as we add $e_i'$ the vertex $c_{i2}'$ contributes nothing to the weight, whilst all other vertices contribute at most one. 
We then reveal $g_i$ for each $i=1,2,\ldots,2k-3$. Again we can conclude that the weight of $g_i$ 
according to this edge exposure is at most $k-1$ as no edges containing $c_{ik}$ or $c'_{i1}$ have been revealed yet and so they contribute nothing to the weight, 
whilst $c_{i*}$ contributes two and all other vertices contribute at most one. Indeed all other vertices in the edge have degree two and so can contribute no more than one to the weight. 
We can then reveal  $f_i$ and $f_i'$ for $i=1,2,\ldots,2k-3$, each of which has weight two, given by the two vertices in the edge which have previously featured. 
Finally, we reveal $h_i$ for $i=1,2,\ldots,2k-4$ which again has weight two given by its two degree two vertices. This gives an edge exposure with degeneracy $k-1$.

\begin{figure}
    \centering
  \includegraphics[scale=.8]{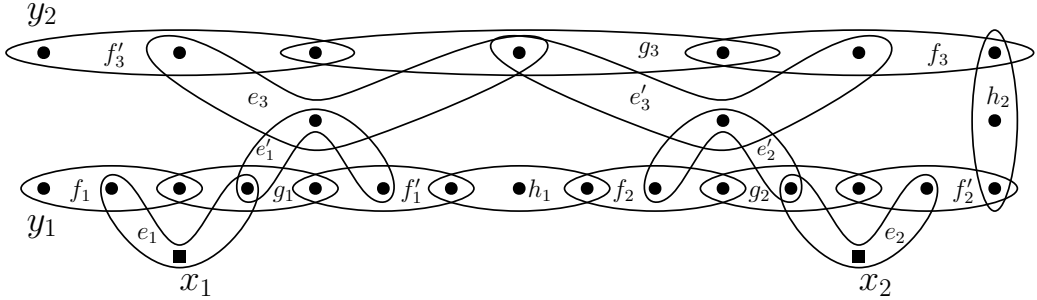}
    \caption{A three-uniform path absorber.}
    \label{fig:three_unif}
\end{figure}
\end{proof}

\subsection{The template}\label{sec:template}

We look to find many (path) absorbers in our host graph and 
%We will use absorbers as the building blocks for an absorbing structure which will allow us to finish an almost spanning tiling in many different ways. 
the relative positions of these absorbers will be determined by an auxiliary hypergraph which we call a template.  

\begin{definition} \label{templatedefn}
An $(r,m)$-\emph{template}  $T$ is an $r$-uniform, $r$-partite  hypergraph  with vertex parts $V(T)=Y_0\cup Y_1 \cup \ldots \cup Y_{r-1}$ 
of sizes  $|Y_0|=4m$, $|Y_1|=\dots=|Y_{r-1}|=3m$ such that the following holds. There exists a subset $Z\subset Y_0$, called the \emph{flexible set}, of size $|Z|=2m$ with the property that for any ${Z}'\subset Z$ of size $m$ the induced hypergraph 
$T[V\setminus{Z}']$ has a perfect matching.
\end{definition}
 As $T$ is $r$-uniform and $r$-partite  there is no confusion in considering edges of $T$ as sets or as ordered tuples and we will switch between these viewpoints throughout.
 For large enough $m$ the existence of a $(2,m)$-template %with parts $(Y_0\sqcup Y_1)$ 
  with maximum vertex-degree $\Delta_1(T)\leq 40$ was  proven by Montgomery~\cite{M14a,montgomery2019spanning} using a probabilistic argument. %In \cite{han2019finding} Kohayakawa, Person and the 
 %last two authors  showed how to construct $(r,m)$-templates of maximum degree $\Delta_1(T)\leq 68042$ in polynomial time for sufficiently large $m$.   
  These easily extend to $(r,m)$-templates by taking $r-2$ disjoint copies $Y_2,\dots,Y_{r-1}$ of $Y_1$ and adding to each edge  $ab\in (Y_0\times Y_1)\cap E(T)$   the copies  
$b^i\in Y_i$ of $b\in Y_1$, $i\in\{2,\dots,r-1\}$ to make  it $r$-uniform. This yields the following.

%It is then easy to adapt this template to templates of higher uniformity. 
%in an argument using the absorbing method to show the existence of certain trees in a Erd\H{o}s-R\'enyi random graph. 
%The existence of such a graph for large enough $m$, is achieved probabilistically by creating a random graph model and showing that after the deletion of vertices we satisfy Hall's condition with positive probability. It is then easy to adapt this template to templates of higher uniformity. 
%This was done by Kwan in \cite[Lemma 3.2]{kwan} for the 3-uniform case to create an absorbing structure for perfect matchings in random Steiner triple systems. We follow the same argument here.

\begin{lemma} \label{lem:template}
For an  $r\geq 2$  and large enough $m$ there is an $(r,m)$-template $T$ with $\Delta_1(T)\leq 40$.\qed
\end{lemma}

Templates $T$  are defined so that they are flexible with respect to perfect matchings in $T$, reflected by the existence of the flexible set $Z\subset V(T)$.
Let $T$ with $V(T)\subset V(H)$ be a suitable chosen template. By combining this  with the notions of absorbers from Section~\ref{subsec:factorabsorbers} 
and Section~\ref{subsec:pathabsorbers}, 
and that of $T$-compatible families  from  Section~\ref{subsec:compatiblefams},
one can transform $V(T)$  into a set $A\supset V(T)$ 
such that $Z\subset A$ is flexible with respect to 
 the desired spanning structure, which means that for each $Z'\subset Z$ of size $m$ the induced $k$-graph 
 $H[A\setminus Z']$ contains an $F$-factor or {a} spanning loose {path} with end vertices independent of $Z'$. 
 This is the key property in our proof of the main theorems.

\section{Proof of Theorem~\ref{thm:Ftilings} and Theorem~\ref{thm:hamcyc}}\label{sec:proofthms}
We prove both theorems,  Theorem~\ref{thm:Ftilings} and Theorem~\ref{thm:hamcyc}, at once following the outline from  Section~\ref{sec:outline}. 
The proof essentially consists of two steps, encapsulated by Claim~\ref{claim:2} and Claim~\ref{claim:3}.
The first, Claim~\ref{claim:2}, concerns the construction of the flexible and absorbing sets $Z\subset A\subset V(H)$  as explained in
the outline given in Section~\ref{sec:outline} and in the last paragraph of the previous section. 
The second, Claim~\ref{claim:3}, then makes use of the flexibility of $Z$ to construct a spanning structure in $H[(V\setminus A)\cup Z']$ for some suitable $Z'\subset Z$.  
Combining this with 
the spanning structure in $H[A\setminus Z']$, which exists by the flexibility of~$Z$, we obtain the desired spanning structure of~$H$. 

Before delving into the proof let us first
expand the outline of the construction of the absorbing set from Claim~\ref{claim:2} by some technical details
and explain the role that  Claim~\ref{claim:1} {plays}.
To be able to make use of the flexible set $Z\subset V(H)$ we require 
that essentially all vertices have high degree into $Z$,  including those in $Z$. 
Due to the necessity to connect paths we require slightly more  in the {Hamilton cycle case}, 
namely that these vertices have high degree, e.g., 
into sets  $Z_1$ and $Z_2$ of a partition of $Z$ and to another set $W$, which we use within the claim  to connect the many small path absorbers to one long path{.} 
One way to ensure this degree condition is to simply choose a slightly larger set and remove vertices with bad degree, of which there are few due to the pseudo-randomness. 
Using Claim~\ref{claim:1} these vertices can then be covered by a small $F$-factor or a loose {path}, respectively, giving rise to the additional  set $U$ in Claim~\ref{claim:2}. 
Claim~\ref{claim:1}
will also become handy in the proof of Claim~\ref{claim:3}. We remark that the proof for $F$-factors (Theorem~\ref{thm:Ftilings}) is simpler than that of Hamilton cycles (Theorem~\ref{thm:hamcyc}). Indeed certain technicalities that arise for Hamilton cycles (e.g.\ the sets $Z_1$ and $Z_2$ and the constant $\alpha$) can be ignored. The reader may therefore prefer to concentrate just on $F$-factors on first read.

\begin{proof}[Proofs of Theorem~\ref{thm:Ftilings} and~\ref{thm:hamcyc}]
Let $k$ and $c$ be given as in the theorems. For Theorem~\ref{thm:Ftilings} we further have the parameter $f\geq k$. 
Let $\Delta=40$ and for convenience choose $\gamma\ll\beta\ll\alpha\ll 1/\Delta,1/f$. We may also assume that $c\ll\gamma$ since it only appears in lower bounds.
Apply Lemma~\ref{lem:absorber} several times, each time with parameters $k$, $\Delta$, and $c_{\ref{lem:absorber}}=c'\ll c$, and for each
$f_{\ref{lem:absorber}}\in\{9k^2-24k+16,f^2-f\}$, and each $r_{\ref{lem:absorber}}\in\{k-1,f\}$,  to obtain $\eps'$, the minimum $\eps_{\ref{lem:absorber}}$ 
over all choices of $f_{\ref{lem:absorber}}$ and $r_{\ref{lem:absorber}}$.
Choose $\eps\ll \eps'$ and $n_0$ sufficiently large. Thus 
we may assume the following hierarchy of constants \[1/n_0\ll\eps\ll\eps'\ll c'\ll c\ll\gamma\ll \beta\ll\alpha\ll1/\Delta,1/f.\]
For Theorem~\ref{thm:Ftilings} let a linear $k$-graph $F$ on $f$ vertices be given. We define 
\[\lF:=\degen(F)+\max_{e\in E(F)}\sum_{v\in e}\deg_F(v)= \degen(F)+\Delta'(F)+k \qquad \text{and}\qquad \lham=\lham(k):=k-1%\begin{cases} 2&\text{if }k=3\\ k??? &\text{otherwise.} \end{cases}
\]
We apply Lemma~\ref{lem:AF} to obtain an absorber $(A_F,\cX_F)$ of $F$
on $f^2$ vertices, of which $f$  are root vertices {such that $A_F$ has } degeneracy $\degen(A_F)=\ell_F$. 
%Recall the important property of $(A_F,\cX_F)$ that there is a complete and an internal $F$-tiling of $A_F$, the first one covering
%all vertices of $A_F$ and second one covering all non-root vertices. % obtained by removing all root vertices~$\cX$.
Similarly, apply Lemma~\ref{lem:AHC3} %(in case $k=3$) and Lemma~\ref{lem:AHCk} %(in case $k> 3$) 
to obtain a path absorber $(A_\ham,\cX_\ham,y_1,y_2)$ 
on $9k^2-23k+15$ vertices, of which $k-1$ are roots {such that $A_\ham$ has degeneracy $\lham$ .}  %(in case $k=3$) and on $??$ vertices (for $k>3$). %Recall that there is a complete loose path and an internal loose path of $A_\ham$ both with the same ends,
%the first one covering all vertices of $A_\ham$ and the second one covering all non-root vertices.
%Let $H$ be a $k$-graph   with the properties stated in the theorem.

Suppose that $H$ is a $(p,\eps p^\ell,%\frac{\eps}{\log (1/p)}
\eps)$-pseudo-random $k$-graph where $\ell=\lF$ when dealing with Theorem~\ref{thm:Ftilings} and $\ell=\lham$ when dealing
with Theorem~\ref{thm:hamcyc}. We start with the following observation which will be used to include small sets of vertices in substructures of our desired spanning structure. 
\begin{claim}\label{claim:1}
Suppose that $\hat B,\hat X\subset V(H)$ are disjoint sets, $(\hat b_1,\dots, \hat b_t)$ is an ordering of the vertices of $\hat B$, 
$|\hat X|\geq \max\{\gamma n/4, 800 (f^2+9k^2)^2 |\hat B|\}$ and $\deg(\hat b;\hat X)\geq cp|\hat X|^{k-1}$ for each $\hat b\in \hat B$. 
Then there is a set $\hat B\subseteq \hat R\subseteq \hat B\cup \hat X$ such that the following holds:
\begin{description}
\item[$F$-factors] If $\ell= \lF$, then   $H[\hat R]$ contains an $F$-factor and $|\hat R|\leq f |\hat B| +f^2$.
%Moreover, if $t=1$ then $\hat B$ and $\hat X$ do not need to be disjoint.
%In particular, if $|\hat X|\geq \gamma n$ then there is a copy of $F$ in $H[\hat X]$.
\item[Hamcyc] If $\ell=\lham$  then $|\hat R|= 3(t-1)(k-1)+1$ and $H[\hat R]$ contains a loose spanning path with end vertices $\hat b_1$ and  $\hat b_t$ and such that
$\hat b_1,\dots,\hat b_t$ appear  in this order in the path.
%Moreover, if $t=2$ then $\hat B$ and $\hat X$ do not need to be disjoint and $\hat R$ can be chosen to be of size $3k-2$ or $4k-3$.
%Moreover, if $|\hat X|\geq \gamma n$, $s\in\{3,4\}$ and  $\deg(\hat v,\hat X)> cp|\hat X|^{k-1}$ for $\hat v\in\{\hat a,\hat b\}$, 
%then there is a loose path in $H[\hat X]$ with ends $\hat a$ and $\hat b$ of length $s$.
\end{description}
%Moreover, if $|\hat B|=2$, $|\hat X|\geq \gamma n$ and $\deg(\hat b;\hat X)\geq 2cp|\hat X|^{k-1}$ for each $\hat b\in \hat B$ then $\hat B$ and $\hat X$ do not need to be disjoint.
\end{claim}
\begin{proof} 
By Fact~\ref{fact:subset} we know that  $H[\hat B\cup \hat X]$ is $\big(p,\eps' p^\ell,%\frac\eps{\log (1/p)}\big
\eps)$-pseudo-random.
In particular, there are $f-1$ vertices  $v_1,\dots, v_{f-1}\in \hat X$ with $\deg(v_i;\hat X)\geq \frac p2|\hat X|^{k-1}$, $i\in[f-1]$, and in the case of $F$-factors we may
add some (at most $f-1$) of them to $\hat B$ so that $|\hat B|$ is a multiple of $f$. Abusing notation slightly we denote this modified set by $\hat B$. In the Hamilton cycle case we leave 
$(\hat b_1,\dots,\hat b_t)$ unchanged.

In both cases, $\ell=\lF$ and $\ell=\lham$, we apply Lemma~\ref{lem:absorber} with 
$H_{\ref{lem:absorber}}=H[\hat B\cup \hat X]$ and  $Y_{\ref{lem:absorber}}=\hat B$.
\begin{itemize}
\item If $\ell=\lF$ then we choose\footnote{One could equally just cover with copies of $F$, avoiding the use of absorbers here but we prefer to use absorbers for brevity.} $(A_{\ref{lem:absorber}},\cX_{\ref{lem:absorber}})=(A_F,\cX_F)$ and { we fix 
$T_{\ref{lem:absorber}}$  as  an} $f$-uniform perfect matching on~$\hat B$. The lemma then implies that there is a $T_{\ref{lem:absorber}}$-compatible family $\cA_F$ 
of copies of $(A_{\ref{lem:absorber}},\cX_{\ref{lem:absorber}})$, i.e., a family of disjoint copies of $A_F$. Taking the complete $F$-factor in each of these copies yields
 an $F$-factor of $H[\hat R]=H[V(\cA_F)]$ and certainly $\hat B \subset V(\cA_F)$.
\item If $\ell=\lham$ then we let  $T_{\ref{lem:absorber}}$ be the $2$-uniform path 
with the vertex ordering $\hat b_1,\hat b_2,\dots,\hat b_t$, i.e., $\hat b_i\hat b_{i+1}\in E( T_{\ref{lem:absorber}})$ for all $i\in[t-1]$ and we let
 $(A',\cX')$ be the loose path of length\footnote{{The \emph{length} of a path is the number of edges in the path.}} three with the ends being the root vertices. Note that 
 $(A',\cX')$ has 
 edge degeneracy two, thus at most $\lham$. 
 Applying Lemma  \ref{lem:absorber} with $(A_{\ref{lem:absorber}},\cX_{\ref{lem:absorber}})=(A',\cX')$, we obtain a family $\cA'$ which consists of $t-1$ length three loose paths with ends $\hat b_i$ 
 and $\hat b_{i+1}$, $i\in[t-1]$, and which are disjoint otherwise.
 Thus, $\hat R=V(\cA')$ has size $3(t-1)(k-1)+1$ and $\cA'$ forms a spanning loose path with ends $\hat b_1$ and $\hat b_t$ in $H[\hat R]$, as required.
 \end{itemize}
 %Finally, note that in the argument from above, we applied Lemma~\ref{lem:absorber} with constant $c_{\ref{lem:absorber}}=c'\ll c$. 
%Hence, if $t=2$ and $\hat B$ and $\hat X$ are not disjoint then making them disjoint decreases the degrees of the vertices in $\hat B$ by at most a factor of $2$.
\end{proof}

With this auxiliary claim proven we now turn to the construction of the absorbing set.
\begin{claim}\label{claim:2}
Fix $m:=\lceil \beta n\rceil$. 
%There are sets $Z_1\sqcup Z_2=Z\subset A \subset V(H)$ and  $U\subset V(H)$ of sizes $|Z_1|=m-\lceil2\gamma n\rceil$, $|Z_2|=\lceil2\gamma n\rceil$, $|A\cup U|\leq 8f^2\beta n$ 
There are vertex sets $Z_1\cup Z_2=Z\subset A \subset V(H)$ and  $U\subset V(H)$ of sizes $|Z_1|=m+ \lceil \gamma n\rceil$ and $|Z_2|=m-\lceil\gamma n\rceil$, $|A\cup U|\leq 8f^2\beta n$ 
such that 
\begin{equation}
\label{eq:highdegZ}\deg(v; Z_i)>\frac p4|Z_i|^{k-1}\quad \text{for } i=1,2\text{ and any }v\in V\setminus U
\end{equation}
and such that  $Z\subset A\cup U$ satisfies the following flexibility property:
\begin{description}
\item[$F$-factors]  If $\ell=\lF$ then for any $Z'\subset Z$ of size $m$ the subgraph $H[(A\cup U)\setminus Z']$ contains an $F$-factor.
\item[Hamcyc] If $\ell=\lham$ then there are $a_1,a_2\in A\setminus Z$ such that for any $Z'\subset Z$ of size $m$ the subgraph 
$H[(A\cup U)\setminus Z']$ contains a spanning loose {path} with ends $a_1$ and $a_2$.
\end{description}
\end{claim}
\begin{proof}
As explained in the beginning of this section we need to do some preprocessing as to guarantee~\eqref{eq:highdegZ}.
Let $r=f$ in the case of $F$-factors and $r=k-1$ in the case of finding a Hamilton cycle.
Let $s=\lceil 10k f\gamma^{-k}\eps p^\ell n\rceil$.
% and let $m=\lceil\beta n\rceil$. 
We choose disjoint sets $\hat Z_1,\hat Z_2\subset V$ of size $|\hat Z_1|=m+ \lceil \gamma n\rceil+s$, $|\hat Z_2|=m- \lceil \gamma n\rceil+s$ and extend $\hat Z=\hat Z_1\cup\hat Z_2$  to a set $\hat Y_0$ of size
$4m+3s$. Further, we choose disjoint sets  $\hat Y_1,\dots, \hat Y_{r-1}\subset V(H)\setminus \hat Y_0$ each of size $3m+s$.
Let $\hat Y=\hat Y_0\cup\dots\cup \hat Y_{r-1}$, let $W\subset V\setminus \hat Y$ be a set of size $\alpha n$ and let
\[B=\left\{v\in V\colon \deg(v;S)< \frac{p}{2} |S|^{k-1}, \text{for some }S\in\{\hat Z_1,\hat Z_2, W,V\setminus(W\cup \hat Y)\}\right\}.\]
From the  pseudo-randomness of $H$ we infer that  $|B|\leq \eps \gamma^{-(k-1)}p^{\ell}n$ and  thus \[\deg(v;V\setminus B)\geq \deg(v)-|B|n^{k-2}>\frac c2pn^{k-1}\quad\text{for each } v\in V.\]
\begin{description}
\item[$F$-factors] For $\ell=\lF$ an application of Claim~\ref{claim:1} with $\hat B=B$, $\hat X= V\setminus B$ then yields a set $U=\hat R\supset B$ of size $|U|\leq f(|B|+f)<s$ 
such that $H[U]$ contains an $F$-factor.
\item[Hamcyc] For $\ell=\lham$ we choose $u_0, v_0\in V\setminus B$ and let $(u_0,b_1,\dots, b_{|B|},v_0)$ be an ordering of the vertices of $B\cup\{u_0,v_0\}$. 
With this ordering and $\hat X=V\setminus (B\cup\{u_0,v_0\})$ an application of Claim~\ref{claim:1} then yields  
a set $U=\hat R\supset B$ of size $|U|\leq 3k|B|<s$ such that $H[U]$ contains a spanning path with ends $u_0$ and $v_0$.
\end{description}

Consequently, from each of the sets $\hat Z_1$, $\hat Z_2$, $\hat Y_0\setminus \hat Z, \hat Y_1,\dots, \hat Y_{r-1}$,  we can remove a subset of size exactly $s$ 
to obtain sets $Z_1\cup Z_2=Z\subset Y_0, Y_1,\dots  Y_{r-1}$ with $|Z_1|=m+ \lceil \gamma n\rceil$, $|Z_2|=m- \lceil \gamma n\rceil$, $|Y_0|=4m$, $|Y_1|=\dots=|Y_{r-1}|=3m$, all disjoint from $U$.
Let $Y=Y_0\cup\dots\cup Y_{r-1}$, let  $V'=V\setminus U$. By the  definition of $B$, the fact that $B\subset U$ and noting that  $s\cdot n^{k-2}<\frac p4 (\gamma n)^{k-1}$ we have that 
\begin{equation}\label{eq:prepareY}
 \deg(v;S)>  \frac p4 |S|^{k-1} \text{ for any }S\in\{Z_1,Z_2, W,V'\setminus  (W\cup Y)\}\quad\text{ and any   }\quad v\in V'\cup\{u_0,v_0\}.
\end{equation}
In particular,  \eqref{eq:highdegZ} holds and we can now turn to the core of the proof of the claim.

\medskip

Recall that $r=f$ in the  case of $F$-factors and $r=k-1$ in Hamilton cycle case. In both cases
we first apply Lemma~\ref{lem:template} to obtain an $(r,m)$-template $T_r$ with  vertex set $Y=Y_0\cup\dots\cup Y_{r-1}$, maximum degree
 $\Delta_1(T_r)\leq \Delta$ and with the flexible set $Z\subset Y_0$. In particular, there is a perfect matching $M(Z')$ of $T_r[Y\setminus Z']$
 for each set $Z'\subset Z$ of size $m$. %These templates we denote by $T_F$ and~$T_\ham$, respectively.
Then, apply Lemma~\ref{lem:absorber} %in both cases %with parameters $k$, $f_{\ref{lem:absorber}}=f^2$, $r=f$, $\Delta=40$, $C,c$, 
with  $H_{\ref{lem:absorber}}=H[V'\setminus W]$, $Y_{\ref{lem:absorber}}=Y$,
 $T_{\ref{lem:absorber}}=T_r$ and with $(A_{\ref{lem:absorber}},\cX_{\ref{lem:absorber}})=(A_F,\cX_F)$ in the case of $F$-factors, and with 
$(A_{\ref{lem:absorber}},\cX_{\ref{lem:absorber}})=(A_{\ham},\cX_\ham)$ in the case of Hamilton cycle. Note that all the conditions of Lemma~\ref{lem:absorber} are indeed satisfied. In particular, the degeneracy of $(A_{\ref{lem:absorber}},\cX_{\ref{lem:absorber}})$ is at most $\ell$ (by our choice of $\ell_F$ and $\ell_\ham$) and the maximum 2-degree condition in Lemma~\ref{lem:absorber} is void for us here as $\ell\geq k-1$. 
This yields a $T_f$-compatible family $\cA_F=\{(A_e,e)\}_{e\in E(T_f)}$ of copies of $(A_F,\cX_F)$
with $Y\subset V(\cA)\subset V'\setminus W$ in the first case and a  $T_{k-1}$-compatible family 
$\cA_\ham=\{(A_e,e,u_e,v_e)\}_{e\in E({T_{k-1})}}$\footnote{Recall that $u_e, v_e\notin e$ and thus they are distinct for all $e\in E({T_{k-1}})$.} of copies of $(A_{\ham},\cX_\ham,y_1,y_2)$ with $Y\subset V(\cA_\ham)\subset V'\setminus W$
in the second. % In the following we show how (the vertex set of) these families form the central part of the set $A$.
In particular, by the defining property of $(A_F,\cX_F)$  we infer that  in any $(A_e,e)\in \cA_F$
there is a complete 
$F$-factor, which covers all of $V(A_e)$, and an internal $F$-factor, which covers $V(A_e)\setminus e$. 
Similarly, in $(A_e,e,u_e,v_e)\in\cA_\ham$ there is a complete loose path,  which covers $V(A_{e})$ and an internal loose path, which covers $V(A_{e})\setminus e$,
 both with the same end vertices~$u_e$ and~$v_e$.  Moreover, being $T_r$-compatible  any two copies $A_e$ and $A_{e'}$ in $\cA_F$ 
 (in $\cA_\ham$, respectively) are disjoint if $e$ and $e'$ are.
 
Together with the flexibility of $Z\subset Y$ with respect to the template $T_f$ we now
 easily establish the flexibility of $Z\subset V(\cA_F)\cup U$, which thus concludes the proof for the case of $F$-factors by setting $A=V(\cA_F)$. 
Indeed,  let $Z'\subset Z$ of size~$m$ be given and let $M(Z')\subset E(T_f)$ denote  a perfect matching  of $Y\setminus Z' \subset A$.
 By taking a complete $F$-factor of $A_e$ if $e\in M(Z')$  while taking an internal $F$-factor of $A_e$
if $e\in E(T_f)\setminus M(Z')$ we obtain an $F$-factor of $H[A\setminus Z']$. Together with the $F$-factor of $H[U]$ and $A\cap U=\emptyset$
this yields an $F$-factor of $H[(A\cup U)\setminus Z']$, as required.

By the same argument we obtain in  the { Hamilton cycle case} that 
for any set $Z'\subset Z$ of size~$m$ there is a collection of $e(T_{k-1})+1$  loose paths {with} fixed end vertices independent of $Z'$, 
and which {span} the vertices of $H[(V(\cA_\ham)\cup U)\setminus Z']$.
%one path for $H[U]$ with end vertices $u_0$ and $v_0$, and one for each $(A_e,e,u_e,v_e)\in \cA_\ham$ with end vertices $u_e$ and $v_e$. 
 Indeed, this follows by considering  a perfect matching $M(Z')$ of $T_{k-1}[Y\setminus Z']$ for a given set $Z'\subset Z$ of size $m$ and
taking in each $(A_e,e,u_e,v_e)\in \cA_\ham$  
 the complete path if $e\in M(Z')$ and the internal path if $e\in E(T_{k-1})\setminus M(Z')$, both having end vertices $u_e$ and $v_e$. Together with   
  the loose spanning path in $H[U]$ with end vertices $u_0$ and $v_0$ we obtain
the required collection of $e(T_{k-1})+1$  loose paths. Thus, to obtain $A$ and the flexibility of $Z\subset A\cup U$ it is left to connect the end vertices of these paths to obtain {a}
long path.

Let $(A_{e_1},\dots, A_{e_t})$ be an ordering of the elements of $\cA_\ham$. %then
 %each $A_{e_i}$ contains two loose paths, the complete path covering all of $V(A_{e_i})$
%and the internal path covering all of $V(A_{e_i})\setminus e_i$, both of which having the same end vertices, say, 
% and for $i\in[t]$, let $u_i$ and $v_i$ be the end vertices of $A_{e_i}$, i.e., the images of  $y_1$ and $y_2$.
As  the end vertices {$u_i:=u_{e_i}$} and  {$v_i:=v_{e_i}$} of $A_{e_i}$, $i\in[t]$, are all contained in $ V'$  we can make use of~\eqref{eq:prepareY} (with $S=W$) to
apply Claim~\ref{claim:1} with the ordering $(\hat b_1,\hat b_2,\dots, \hat b_{2t})=(v_0,u_1, v_1, u_2, v_2,\dots, u_t)$, where $t < \Delta_1(T_{k-1})\cdot |{Y_1}|\leq 3 \Delta m$ and
$\hat X=W$ to find a loose path connecting these vertices as given in the order. From this path we only keep the connecting paths $P_{i+1}$ between $v_i$ and $u_{i+1}$, $i=0,\dots, t-1$, discarding
 each of the loose paths between $u_i$ and $v_i$.  
 Now let $A=V(\cA_\ham)\cup V(P_1)\cup\dots\cup V(P_t)$ and let  $a_1=u_0$ and $a_2=v_t$. Together with the argument from above we conclude that
 $Z\subset A\cup U$ has the  desired flexibility property.
 \end{proof}

Using Claim \ref{claim:2}, we can now prove the following claim which will conclude the proof.

\begin{claim}\label{claim:3}
Let $V''=V\setminus (A\cup U)$ then there is a set $Z'\subset Z$ of size $m$ such that the following holds:
\begin{description}
\item[$F$-factors] If $\ell=\lF$ then $H[V''\cup Z']$ contains an $F$-factor.
\item[Hamcyc] If $\ell=\lham$ then $H[V''\cup Z'\cup\{a_1,a_2\}]$ contains a spanning loose   path with ends $a_1$ and $a_2$.
\end{description}
\end{claim}
Before proving the claim note that it  implies the theorems. Indeed, if $\ell=\lF$ then we use Claim~\ref{claim:3} and choose an $F$-factor
of $H[V''\cup Z']$, for some $Z'\subset Z$ of size $m$. Claim~\ref{claim:2} then guarantees that there is an $F$-factor
of $H[(A\cup U)\setminus Z']$ which thus yields an $F$-factor of $H$, as $V=V''\cup A \cup U$. For $\ell=\lham$ 
we take a loose Hamilton  path of $H[V''\cup Z'\cup\{a_1,a_2\}]$ with ends $a_1$ and $a_2$. By Claim~\ref{claim:2}
there is a loose Hamilton  path of $H[(A\cup U)\setminus Z']$ with the same end vertices which thus yields a Hamilton cycle of $H$.
\end{proof}

\begin{proof}[Proof of Claim~\ref{claim:3}]

 Consider first the $F$-factor case.
Let $R\subset V''$ be the largest set such that $H[R]$ contains an $F$-factor and let $L=V''\setminus R$ be the set of uncovered vertices. Suppose $|L|\geq \gamma n$, then there is 
a vertex $v\in L$ with $\deg(v; L)>\frac p2|L|^{k-1}$ and by applying Claim~\ref{claim:1}  with $\hat B=\{v\}$ and $\hat X=L\setminus \{v\}$ we  find a copy of $F$ in $L$, 
 contradicting the maximality of $R$. Thus $|L|<\gamma n$
and the claim follows (by setting $R\cup S=V''\cup Z'$) once we have shown that there is a set $S$ such that
\begin{enumerate}
\item $L\subset S\subset L\cup Z$ and $H[S]$ contains an  $F$-factor,
\item $Z\setminus S$ has size $m$.
\end{enumerate}
To find $S$ consider first  a smallest set $S_1$ which satisfies the first property.    Such a set exists since we can apply Claim~\ref{claim:1} with 
$\hat B=L$ and $\hat X=Z$, noting that the assumptions are met due to~\eqref{eq:highdegZ} and the fact that  $L\subset V\setminus U$. 
Thus $S_1$ exists and $|S_1|\leq f(|L|+f)<m$.  
Now let $S_2\subset Z\setminus S_1$ be the largest set such that $H[S_2]$ contains an $F$-factor.
By the same argument as in the previous paragraph  $|Z\setminus (S_2\cup S_1)|<\gamma n$.
%, i.e., $\tilde Z=Z\setminus \tilde S$ has size $|\tilde Z|<\gamma n$. 
Finally, due to Claim~\ref{claim:2} we have \[|V''|+m=|V|-(|A\cup U|-m)\in f\NN,\] and  
\[|V''|+2m=|V''\cup Z|=|R|+|S_1|+|S_2|+|Z\setminus  (S_1\cup S_2)|.\] 
This yields $|Z\setminus  (S_1\cup S_2)|-m\in f\NN$ and 
therefore we can remove copies of $F$ from $S_2$ to obtain $S_1\subset S\subset S_1\cup S_2$ with the required properties.
%Thus by removing copies of $F$ from $\tilde S$ which do not contain vertices from $L$ we obtain the required set $S$.

%Thus we can choose
%a set $S$ which satisfies the first and is closest possible to satisfying the second property. Let $\tilde Z=Z\setminus S$.

%Note that for any $L\subset V'$ of size $|L|>\gamma n$ there is a vertex $v\in L$ with $\deg(v; L)>\frac p2|L|^{k-1}$ and thus we can find a copy of $F$ in $L$ by applying Claim~\ref{claim:1}  with $\hat B=\{v\}$ and $\hat X=L\setminus \{v\}$.
%Thus, by repeatedly choosing vertex disjoint copies of $F$ we obtain a set $R\subset V'$ of size $|R|\geq |V'|-\gamma n$
%such that $H[R]$ contains a perfect $F$-factor and we let $L=V'\setminus R$.
%Due to~\eqref{eq:highdegZ} and $L\subset V\setminus U$ we can apply Claim~\ref{claim:1} with $\hat B=L$ and $\hat X=Z=Z_1\cup Z_2$ to obtain 
%a set $L\subset R'\subset R'\cup Z$ of size $|R'|\leq f^2|L|<m$ and so that $H[R']$ contains a perfect $F$-factor.
%Let $\tilde Z=Z\setminus R'$ which has size $|\tilde Z|=2m-|R'|>m>\gamma n$. Further, we have
%$|V'|+m=|V|-(|A\cup U|-m)\in f\NN$, and $|V'|+2m=|V'\cup Z|=|R|+|R'|+|\tilde Z|$ which yields $|\tilde Z|-m\in f\NN$.
 %Thus, by arguing as above we can cover all but $m$ vertices of $\tilde Z$ by vertex disjoint copies of $F$ and the claim follows.
 \vspace{2mm}
  Let us now turn to considering the Hamilton loose path. Here the argument is very similar to the above but it is slightly more delicate as we have to connect the loose paths that we find to one. For this, we use the partition of $Z$ into $Z_1\cup Z_2$ as in Claim \ref{claim:2}, and carry out the argument using only vertices from $Z_1$, reserving the vertices of $Z_2$ to connect the paths in the very last step. The details follow. 
  
  In $V''$ we choose the largest set $R\subset V''$ with the property that $H[R]$ contains a loose Hamilton path with one of its end vertices, say, $a\in R$ satisfying $\deg(a;V''\setminus R)>2cpn^{k-1}$.
 Let $L=V''\setminus R$ and suppose that $|L|\geq\gamma n$. Then there is  a vertex $b\in L$ with $\deg(b;L)>4cpn^{k-1}$.
 Applying Claim~\ref{claim:1} with  $\hat B=\{a,b\}$ and with $\hat X=L\setminus\{b\}$ we  then find a path of length three in $L\cup\{a\}$ connecting $a$ and~$b$,
which thus yields a contradiction to the maximality of $R$. Thus $|L|<\gamma n$.
 Next, we claim that there is a set $S$ such that 
\begin{enumerate}\item $L\subset S\subset L\cup Z_1$ and $H[S]$ has a spanning subgraph consisting of two vertex disjoint 
loose paths,
\item  $|Z\setminus S|=m+12(k-1)-4$.
\end{enumerate}
To find $S$  consider first  a smallest set  $S_1$ with $L\subset S_1\subset L\cup Z_1$ and a largest set $S_2\subset Z_1\setminus S_1$
such that $H[S_1]$ and $H[S_2]$ both contain a Hamilton path.
Due to~\eqref{eq:highdegZ} and the fact that  $L\subset V\setminus U$ we can apply Claim~\ref{claim:1}  
with an arbitrary ordering of the vertices of $\hat B=L$ and $\hat X=Z_1$. This shows that $S_1$ exists and $|S_1|\leq 3k|L|<m-12(k-1)$. 
Further, using the same argument which was used to find $R$ above, we have that  $|Z_1\setminus S_2|<\gamma n$, thus $|Z\setminus (S_1\cup S_2)|<\gamma n+|Z_2|\le m$. 
%Again we will remove vertices from $S_2$ to obtain the required set $S$.
Note that  $|S_i|\equiv1\mod (k-1)$, $i=1,2$ and the same holds for $|R|$ and also for $(|A\cup U|-m)$ due to Claim~\ref{claim:2}. With $|V|\in (k-1)\NN$ this yields
\[|V''|+2m=|V|+m-(|A\cup U|-m)\equiv m-1\mod (k-1)\] and moreover we have \[|V''|+2m=|V''\cup Z|=|R|+|S_1|+|S_2|+|Z\setminus (S_1\cup S_2)|.\]
This yields $|Z\setminus (S_1\cup S_2)|\equiv m-4\mod (k-1)$ and therefore by shortening the path in $S_2$ we can enlarge $|Z\setminus (S_1\cup S_2)|$ and thus choose a set
$S_1\subset S\subset S_1\cup S_2$ with the required properties.

Finally, let $(b_1, b_2)$, $(c_1,c_2)$ and $(d_1,d_2)$ denote the ends of a Hamilton path in $H[R]$ and the  two paths  in $H[S]$  which cover all of $S$.
Note that  these vertices are contained in $V\setminus U$.
Hence, by~\eqref{eq:highdegZ} we can apply Claim~\ref{claim:1} with $(a_2,b_1)$ and $\hat X=Z_2$  
and find a set $R_{a_2,b_1}\subset Z_2$ of size $3(k-1)-1$
which connect $a_2$ and $b_1$ by a loose path. We repeat the argument with $(b_2,c_1)$ and $\hat X=Z_2\setminus R_{a_2,b_1}$ to find 
 $3(k-1)-1$ vertices in $Z_2\setminus R_{a_2,b_1}$ to
connect $b_2$ and $c_1$ and in the same manner connect $(b_2,c_1)$, $(c_2,d_1)$ and $(d_2,a_1)$.
This yields a loose path with ends $a_1$ and $a_2$ which 
covers all but $|Z\setminus S|-12(k-1)+4=m$  vertices of $V''\cup Z$, and the claim follows.
\end{proof}

\bibliography{refs}

\end{document}